\documentclass[11pt,a4paper]{amsart}
\usepackage{amsmath,amssymb,mathrsfs,enumerate,booktabs,hyperref}
\usepackage[alphabetic]{amsrefs}
\usepackage{tikz}
\usetikzlibrary{cd}

\input{preamble}
\mytheoremstyleenglish

\title
{Canonical coverings of Enriques surfaces in characteristic $2$}

\author{Yuya Matsumoto}
\date{2021/01/31}
\address{\tusaddressfull}
\email{\gmail}
\email{\tusmail}
\thanks{This work was supported by JSPS KAKENHI Grant Number 16K17560 and 20K14296.}

\begin{document}

\begin{abstract}
Let $\bar{Y}$ be a normal surface that is the canonical $\mu_2$- or $\alpha_2$-covering of a classical or supersingular Enriques surface in characteristic $2$.
We determine all possible configurations of singularities on $\bar{Y}$,
and for each configuration we describe which type of Enriques surfaces (classical or supersingular) appear
as quotients of $\bar{Y}$.
\end{abstract}

\maketitle

\section{Introduction}

Let $X$ be an Enriques surface
over an algebraically closed field $k$
(see Section \ref{subsec:basic} for the definition). 
It is known that the torsion part $\Pictau_X$ of the Picard scheme of $X$ is a finite group scheme of order $2$,
and thus there is a canonical $G$-covering $\bar{Y} \to X$,
where $G := \sHom(\Pictau_X, \Gm)$ is the Cartier dual of $\Pictau_X$.
If $\charac k \neq 2$, then $\Pictau_X$ and $G$ are both isomorphic to $\bZ/2\bZ$,
the covering is a finite \'etale $\bZ/2\bZ$-covering, and $\bar{Y}$ is a smooth K3 surface.
If $\charac k = 2$, then the situation is more complicated: 
there are three possibilities for $G$, namely $\bZ/2\bZ$, $\mu_2$, and $\alpha_2$.
In this paper we study the \emph{classical} and \emph{supersingular Enriques surfaces} in characteristic $2$, 
that is, $G = \mu_2$ or $G = \alpha_2$ respectively.
In these cases the canonical $G$-covering $\bar{Y}$ is always singular,
and $\bar{Y}$ may not be even birational to a K3 surface.

In this paper, we restrict our attention to the case where $\bar{Y}$ is normal,
and we discuss two problems:
to determine the possible configurations of singularities on $\bar{Y}$,
and to describe all Enriques quotients of $\bar{Y}$. 

\subsection{Previous research} \label{subsec:previous results}

For the singularities, the following is known. 
\begin{thm}[\cite{Cossec--Dolgachev:enriques}*{Proposition 1.3.1 and Theorem 1.3.1}] \label{thm:classification of canonical covering}
Let $\bar{Y}$ be the canonical covering of a classical or supersingular Enriques surface in characteristic $2$.
Then $\bar{Y}$ is K3-like (see Section \ref{subsec:basic}).
If $\bar{Y}$ is normal, then one of the following holds:
	\begin{itemize}
		\item $\bar{Y}$ has only rational double points (RDPs) as singularities. 
			In this case $\bar{Y}$ is an RDP K3 surface (Section \ref{subsec:basic}).
		\item $\bar{Y}$ has only isolated singularities and contains a non-RDP singularity. 
			In this case there is exactly one non-RDP singularity and it is an elliptic double point (EDP), and $\bar{Y}$ is a rational surface.
	\end{itemize}
\end{thm}

(However, the proof of \cite{Cossec--Dolgachev:enriques}*{Proposition 1.3.1} contains a gap.
See Remark \ref{rem:CD131}.)

The following description is given by Ekedahl--Hyland--Shepherd-Barron.
\begin{thm}[Ekedahl--Hyland--Shepherd-Barron \cite{Ekedahl--Hyland--Shepherd-Barron}*{Corollary 3.7(3) and Corollary 6.16}] \label{thm:EHSB}
	Let $\bar{Y}$ be as in Theorem \ref{thm:classification of canonical covering}
	and assume $\bar{Y}$ is normal.
	Then,
	\begin{itemize}
	\item 
	The tangent sheaf $T_{\bar{Y}}$ is free.
	\item 
	If $\Sing(\bar{Y})$ consists only of RDPs, then $\Sing(\bar{Y})$ is one of 
	\[ 12 A_1, \,  8 A_1 + D_4^0, \,  6 A_1 + D_6^0, \,  5 A_1 + E_7^0, \,  3 D_4^0, \,  D_4^0 + D_8^0, \,  D_4^0 + E_8^0, \,  D_{12}^0 . \]
	\end{itemize}
\end{thm}

They also claimed that all global derivations $D \in H^0(\bar{Y}, T_{\bar{Y}})$ are $p$-closed
(\cite{Ekedahl--Hyland--Shepherd-Barron}*{Corollary 7.3}),
but their proof covers only the generic case.

Recently Schr\"oer proved the following results.
\begin{thm}[Schr\"oer \cite{Schroer:K3-like}*{Theorems 6.3--6.4 and Sections 13--15}] \label{thm:Schroer}
	Let $\bar{Y}$ be as in Theorem \ref{thm:classification of canonical covering}
	and assume $\bar{Y}$ is normal.
\begin{enumerate}
	\item 
	The quotient of $\bar{Y}$ by a global fixed-point-free derivation
	that is either of multiplicative type or of additive type
	is an Enriques surface.
	\item \label{thm:Schroer:family}
	Under a mild assumption on $\bar{Y}$, 
	all global derivations $D \in H^0(\bar{Y}, T_{\bar{Y}})$ are $p$-closed, 
	and most of them (those belonging to the complement of finitely many lines) are fixed-point-free.
	Hence $\bar{Y}$ admits a $1$-dimensional family of Enriques quotients
	parametrized by a nonempty open subscheme of $\bP(H^0(\bar{Y}, T_{\bar{Y}})) \cong \bP^1$.
	\item \label{thm:Schroer:EDP example}
	There exists an example of $\bar{Y}$ with an EDP.
	\item 
	If $\Sing(\bar{Y})$ contains an EDP,
	then $\Sing(\bar{Y})$ consists precisely of that point.
	\item \label{thm:Schroer:construction}
	There is a method of a construction, 
	from a given rational elliptic surface $J \to \bP^1$ satisfying certain assumptions on singular fibers, 
	of a $J$-torsor $X \to \bP^1$ that is an Enriques surface 
	whose canonical covering $\bar{Y}$ is birational to the Frobenius base change $\thpower{J}{2/\bP^1}$
	and moreover having the same type of singularities as $\thpower{J}{2/\bP^1}$.
\end{enumerate}
\end{thm}

Ekedahl--Hyland--Shepherd-Barron did not show whether every configuration in Theorem \ref{thm:EHSB} is actually possible. 
The method of Theorem \ref{thm:Schroer}(\ref{thm:Schroer:construction})
applied to various rational elliptic surfaces
would construct examples for all configurations in Theorem \ref{thm:EHSB},
but this is not explicitly mentioned, 
and classical and supersingular surfaces are not explicitly distinguished.

Katsura--Kondo (\cite{Katsura--Kondo:1-dimensional}*{Section 4} and \cite{Katsura--Kondo:Enriques finite group}*{Section 3}) (resp.\ Kondo (\cite{Kondo:coveredby1}*{Section 3}))
described the families of Enriques quotients
of two (resp.\ one) explicit examples of canonical coverings $\bar{Y}$ with $\Sing(\bar{Y}) = 12 A_1$ (resp.\ $\Sing(\bar{Y}) = 8 A_1 + D_4^0$).
They consist of both classical and supersingular Enriques quotients (resp.\ only classical ones).

\subsection{Our results} \label{subsec:our results}

Now we shall state the main results of this paper.
Let $\bar{Y} \to X$ be as above and assume $\bar{Y}$ is normal.
We show that the conclusion of Theorem \ref{thm:Schroer}(\ref{thm:Schroer:family}) holds unconditionally.
We describe the ($2$-dimensional) restricted Lie algebra $H^0(\bar{Y}, T_{\bar{Y}})$
and the ($1$-dimensional family of) Enriques quotients of $\bar{Y}$.
The answers depend on the configuration of singularities on $\bar{Y}$ and,
perhaps surprisingly,
if $\Sing(\bar{Y})$ is other than $12 A_1$, 
then the Enriques quotients of $\bar{Y}$ are either all classical or all supersingular.
We also determine which configurations of singularities actually occur.
It turns out that in the RDP case every configuration in Theorem \ref{thm:EHSB} is possible.
In the EDP case there is only one possible configuration:
one EDP of type $\Ella$ 
(the singularity defined by $z^2 + x^3 + y^7 = 0$, see Section \ref{subsec:RDP EDP} for the precise definition and properties).

\begin{thm} \label{thm:main3}
	Let $\bar{Y}$ be a normal surface that is the canonical covering of some classical or supersingular Enriques surface in characteristic $p = 2$.
	(Then, by Theorem \ref{thm:EHSB}, the tangent sheaf $T_{\bar{Y}}$ is free and hence 
	$\fg := H^0(\bar{Y}, T_{\bar{Y}})$ is $2$-dimensional.)

	Then all element $D \in \fg$ are $p$-closed 
	and most of them (those belonging to the complement of finitely many lines) are fixed-point-free.
	Hence, as above, $\bar{Y}$ admits a $1$-dimensional family of Enriques quotients,
	parametrized by a nonempty open subscheme of $\bP(\fg) \cong \bP^1$.
	Moreover, according to the singularities of $\bar{Y}$, 
	the following assertions hold.
	\begin{enumerate}
		\item \label{case:main:A}
		Suppose $\Sing(\bar{Y})$ is $12 A_1$.
		Then the subset $\fl := \set{D \in \fg \mid D^p = 0}$ is a line,
		and each nonzero element of $\fl$ is fixed-point-free.
		Hence there is exactly one Enriques quotient of $\bar{Y}$ that is supersingular
		and all other Enriques quotients are classical.
		\item \label{case:main:AD}
		Suppose $\Sing(\bar{Y})$ is one of 
		$  8 A_1 + D_4^0$, $6 A_1 + D_6^0$, or $5 A_1 + E_7^0$.
		Then the subset $\fl := \set{D \in \fg \mid D^p = 0}$ is a line,
		and each nonzero element of $\fl$ is not fixed-point-free.
		Hence all Enriques quotients of $\bar{Y}$ are classical.
		\item \label{case:main:DE}
		Suppose $\Sing(\bar{Y})$ is one of 
		$  3 D_4^0$, $D_4^0 + D_8^0$, $D_4^0 + E_8^0$, or $D_{12}^0$.
		Then all $D \in \fg$ satisfy $D^p = 0$.
		Hence all Enriques quotients of $\bar{Y}$ are supersingular.
		\item \label{case:main:EDP}
		Suppose $\Sing(\bar{Y})$ contains an EDP.
		Then the EDP is of type $\Ella$ and this is the only singularity of $\bar{Y}$,
		and all $D \in \fg$ satisfy $D^p = 0$.
		Hence all Enriques quotients of $\bar{Y}$ are supersingular.
	\end{enumerate}
	In cases (\ref{case:main:A}) and (\ref{case:main:AD}),
	the restricted Lie algebra
	$\fg$ is non-abelian and the image of the bracket is $\fl$.
	In cases (\ref{case:main:DE}) and (\ref{case:main:EDP}),
	$\fg$ is abelian.
\end{thm}
\begin{thm} \label{thm:main4}
	The $9$ configurations of $\Sing(\bar{Y})$ mentioned in Theorem \ref{thm:main3} 
	are precisely the ones that can occur for the normal canonical coverings of 
	classical or supersingular Enriques surfaces in characteristic $2$.
\end{thm}

As explained above,
Theorem \ref{thm:main4} follows implicitly from Theorem \ref{thm:Schroer}(\ref{thm:Schroer:construction}) of Schr\"oer in the RDP cases,
and is explicitly proved by Schr\"oer (Theorem \ref{thm:Schroer}(\ref{thm:Schroer:EDP example})) in the EDP case (modulo the assertion that the EDP is $E_{12}$).

\begin{cor}
The possible configurations of singularities on the normal canonical coverings of classical (resp.\ supersingular) Enriques surfaces in characteristic $2$ are
	\begin{gather*}
	  12 A_1, \, 8 A_1 + D_4^0, \, 6 A_1 + D_6^0, \, \text{and } 5 A_1 + E_7^0  \\
    \text{(resp. }
	  12 A_1, \, 3 D_4^0, \, D_4^0 + D_8^0, \, D_4^0 + E_8^0, \, D_{12}^0, \, \text{and } \Ella 
	\text{).}
	\end{gather*}
\end{cor}

\medskip

This paper is organized as follows.
In Section \ref{sec:prelim}, we introduce some notions and basic facts on 
K3 and Enriques surfaces, derivations, and restricted Lie algebras.
In Section \ref{sec:derivation quotient} we discuss $p$-closed derivation quotients 
of rational double point (RDP) singularities,
elliptic double point (EDP) singularities,
and K3-like surfaces 
(mainly in characteristic $2$).

In Section \ref{sec:proof}
we prove Theorem \ref{thm:main3}. 
Our proof relies on
the previous works of Ekedahl--Hyland--Shepherd-Barron \cite{Ekedahl--Hyland--Shepherd-Barron} and Schr\"oer \cite{Schroer:K3-like},
and techniques from the recent preprint \cite{Matsumoto:k3alphap} of the author
on $\mu_p$- and $\alpha_p$-actions on K3 surfaces.

In Section \ref{sec:examples}, 
we recall the examples of $12 A_1$ and $8 A_1 + D_4^0$ given by Katsura--Kondo and Kondo,
and give examples of the remaining configurations, 
thus proving Theorem \ref{thm:main4}.
Our constructions for the RDP cases are either straight generalizations of Kondo's (for classical cases) or influenced by his (for supersingular cases).
A difference is that our presentation deals with regular derivations on RDP K3 surfaces, 
which would be easier to compute 
than rational derivations on smooth K3 surfaces used in Katsura--Kondo's and Kondo's.
Also, most of our constructions can be viewed as explicit special cases of Schr\"oer's constructions. 

\section{Preliminaries} \label{sec:prelim}

Throughout the paper we work over an algebraically closed field $k$ of characteristic $p \geq 0$.

\subsection{Enriques surfaces and K3-like surfaces} \label{subsec:basic}

A \emph{K3 surface} is a proper smooth surface $X$ with $K_X = 0$ and $H^1(X, \cO_X) = 0$.
An \emph{Enriques surface} is a proper smooth surface $X$ with $K_X$ numerically trivial and $b_2(X) = 10$.
Here $b_i(X) := \dim \Het^i(X, \bQ_l)$ is the $l$-adic Betti number for an auxiliary prime $l \neq \charac k$.

Suppose $X$ is an Enriques surface.
In characteristic $\neq 2$,
we have $K_X \not\sim 0$, $2 K_X \sim 0$, $\Pictau_X \cong \bZ/2\bZ \cong \mu_2$.
Here $\sim$ is the linear equivalence.
In characteristic $2$,
exactly one of the following holds (\cite{Bombieri--Mumford:III}*{Section 3}).
\begin{itemize}
\item $K_X \sim 0$, $\Pictau_X \cong \mu_2$. 
In this case $X$ is called \emph{singular}.
\item $K_X \not\sim 0$, $2 K_X \sim 0$, $\Pictau_X \cong \bZ/2\bZ$. 
In this case $X$ is called \emph{classical}.
\item $K_X \sim 0$, $\Pictau_X \cong \alpha_2$. 
In this case $X$ is called \emph{supersingular}.
\end{itemize}
In any case, the isomorphism 
$\Hfl^1(X, {G}) \cong \Hom(\dual{G}, \Pictau_X)$ of \cite{Schroer:K3-like}*{Proposition 4.1}
(where $G$ is a finite commutative group scheme and $\dual{G} = \sHom(G, \Gm)$ is its Cartier dual)
induces a canonical $\dual{(\Pictau_X)}$-torsor $\bar{Y} \to X$,
which we call the \emph{canonical covering} of $X$.

An \emph{RDP K3 surface} (resp.\ \emph{RDP Enriques surface})
is a proper surface with only RDPs as singularities (if any) whose minimal resolution is a smooth K3 (resp.\ Enriques) surface. 

We say that an RDP Enriques surface is classical or supersingular if its minimal resolution is so.

A \emph{K3-like surface}, following \cite{Bombieri--Mumford:III},
is a proper reduced Gorenstein surface $X$, not necessarily normal,
whose dualizing sheaf $\omega_X$ is isomorphic to $\cO_X$
and satisfying $h^i(X, \cO_X) = 1,0,1$ for $i = 0,1,2$.
Any RDP K3 surface is K3-like.
Any K3-like surface with $b_1 = 0$ is either an RDP K3 surface or a (normal or non-normal) rational surface
by \cite{Cossec--Dolgachev:enriques}*{proof of Theorem 1.3.1}.

A \emph{genus one fibration} on a smooth proper surface $X$
is a morphism $X \to \bP^1$, not necessarily with a section, whose generic fiber is a curve of arithmetic genus one.
It is called an \emph{elliptic fibration} (resp.\ a \emph{quasi-elliptic fibration}) if the generic fiber is a smooth elliptic curve (resp.\ a cuspidal rational curve).
We do not use quasi-elliptic fibrations in this paper.

\begin{prop}[\cite{Cossec--Dolgachev:enriques}*{Theorems 5.7.2 and 5.7.5}] \label{prop:multiple fibers}
	Let $X \to \bP^1$ be a genus one fibration on a classical Enriques surface in characteristic $2$.
	Then there are exactly $2$ multiple fibers,
	and each is either a smooth ordinary elliptic curve or a singular fiber of additive type.
\end{prop}

\subsection{Derivations}

A \emph{(regular) derivation} on a scheme $X$ 
is a $k$-linear endomorphism $D$ of $\cO_X$ 
satisfying $D(ab) = a D(b) + D(a) b$.

The \emph{fixed locus} $\Fix(D)$ of a derivation $D$
is the closed subscheme of $X$ 
corresponding to the ideal $(\Image (D))$ generated by $\Image(D) = \set{D(a) \mid a \in \cO_X}$.
If $X$ is normal and $D \neq 0$, then the \emph{divisorial part} of $\Fix(D)$ is denoted by $\divisorialfix{D}$.

Assume $X$ is a smooth integral surface and $D \neq 0$. 
Then we define the \emph{isolated part} of $\Fix(D)$, denoted $\isolatedfix{D}$, as follows.
If we write $D = f (g \partial/\partial x + h \partial / \partial y)$ with $g,h$ coprime for some local coordinate $x,y$,
then $\divisorialfix{D}$ and $\isolatedfix{D}$ correspond to the ideal $(f)$ and $(g,h)$ respectively.

Suppose for simplicity that $X$ is integral.
Then a \emph{rational derivation} on $X$ 
is a global section of $\Der(\cO_X) \otimes_{\cO_X} k(X)$, 
where $\Der(\cO_X)$ is the sheaf of derivations on $X$.
Thus, a rational derivation is locally of the form $f^{-1} D$ with $f$ a regular function and $D$ a regular derivation.
We extend the notion of divisorial and isolated parts to rational derivations by
$\divisorialfix{f^{-1} D} = \divisorialfix{D} - \divisor(f)$ and $\isolatedfix{f^{-1} D} = \isolatedfix{D}$.

Suppose $\charac k = p > 0$.
A derivation $D$ is said to be \emph{of multiplicative type} (resp.\ \emph{of additive type})
if $D^p = D$ (resp.\ $D^p = 0$).
Such derivations correspond to actions of the group scheme $\mu_p$ (resp.\ $\alpha_p$) on the scheme.
More generally, $D$ is said to be \emph{$p$-closed} if there exists $h \in k(X)$ with $D^p = h D$.

We recall the Rudakov--Shafarevich formula and the Katsura--Takeda formula.
\begin{thm}[Rudakov--Shafarevich \cite{Rudakov--Shafarevich:inseparable}*{Corollary 1 to Proposition 3}]
	Let $D$ be a nonzero $p$-closed rational derivation on a smooth variety $X$ in characteristic $p > 0$.
	Denote by $\map{\pi}{X}{X^D = Y}$ the quotient morphism.
	Then we have
	\[ 
	K_X \sim \pi^* K_Y + (p-1) \divisorialfix{D},
	\]
	where $\sim$ is the linear equivalence.
\end{thm}
\begin{thm}[Katsura--Takeda \cite{Katsura--Takeda:quotients}*{Proposition 2.1}] \label{thm:c2 derivation}
	Let $D$ be a nonzero rational derivation on a smooth proper surface $X$. 
	Then
	\[
	\deg c_2(X) = \deg \isolatedfix{D} - K_X \cdot \divisorialfix{D} - \divisorialfix{D}^2.
	\]
\end{thm}

In characteristic $p = 2$ we have the following corollary of the Rudakov--Shafarevich formula.
\begin{prop}[\cite{Ekedahl--Hyland--Shepherd-Barron}*{Lemma 3.14}] \label{prop:half fix}
	Let $D$ be a nonzero $p$-closed rational derivation on a smooth variety $X$ in characteristic $p = 2$.
	Then $K_X - \divisorialfix{D}$ is divisible by $2$ in $\Pic(X)$.
\end{prop}
\begin{proof}
Let $\map{\pi}{X}{X^D = Y}$ the quotient morphism.
Then the ``dual'' morphism $\map{\pi'}{Y}{\secondpower{X}}$ is purely inseparable of degree $2$,
hence is the quotient by some rational derivation $D'$.
By replacing with a multiple we may assume $D'^2 = 0$.
Then $I := \Image(D')$ is a fractional ideal of $\secondpower{X}$.
By removing a closed subscheme of $X$ of codimension at least $2$ (which does not change the Picard group), we may assume $I$ is principal, thus identified with a divisor $\Delta$.
Then $\divisorialfix{D'} = \pi'^*(\Delta)$.
By the Rudakov--Shafarevich formula, we have
\begin{align*}
K_X - \divisorialfix{D} 
  \sim \pi^* K_Y 
 \sim \pi^* (\divisorialfix{D'} + \pi'^*(K_{\secondpower{X}})) 
 \sim \pi^* \pi'^* (\Delta + K_{\secondpower{X}}), 
\end{align*}
and the image of $\pi^* \pi'^* = F^*$
is divisible by $p$.
\end{proof}

\subsection{Restricted Lie algebras of dimension $2$} \label{subsec:Lie}

Recall that a \emph{restricted Lie algebra} over a field $k$ of characteristic $p > 0$
is a $k$-vector space $\fg$
together with two operation, 
the bracket $\map{[\functorspace, \functorspace]}{\fg \times \fg}{\fg}$
and the $p$-th power map $\map{\pthpower{\functorspace}}{\fg}{\fg}$,
satisfying certain conditions.
An example is $H^0(X, T_X)$ for a scheme $X$,
where the bracket is the usual one ($[D_1, D_2] = D_1 \circ D_2 - D_2 \circ D_1$)
and the $p$-th power $\pthpower{D}$ of $D$ is the $p$-th iterate $D^p = D \circ \dots \circ D$.
(In this example, $\circ$ is defined only on $\sEnd(\cO_X) \supset T_X$, but $[D_1, D_2]$ and $D^p$ belong to $T_X$.)

We say that 
an element $x$ of a restricted Lie algebra $\fg$ 
is \emph{$p$-closed} if it satisfies $\pthpower{x} = \lambda x$ for some scalar $\lambda \in k$,
and that it is \emph{of multiplicative type} (resp.\ \emph{of additive type}) 
if we can take $\lambda \neq 0$ (resp.\ $\lambda = 0$).
We also say that a line $[x] = k x \subset \fg$
generated by a nonzero element $x$
is \emph{$p$-closed}, \emph{of multiplicative type}, or \emph{of additive type}
if it contains a nonzero element with those properties.

Note that if $X$ is proper and $T_{X}$ is free, 
then a line of $\fg = H^0(X, T_X)$ is $p$-closed (in this sense, where the ratio is a scalar) 
if and only if some, equivalently any, nonzero element in the line is $p$-closed (in the sense of Section \ref{subsec:basic}, where the ratio can be any rational function).

\begin{prop}[\cite{Wang:hopfalgebras}*{Proposition A.3}] \label{prop:classification of Lie algebra}
	There are exactly $5$ isomorphism classes of restricted Lie algebras $\fg$ of dimension $2$ 
	(over a fixed algebraically closed field $k$ in characteristic $p > 0$).
	In each case there is a basis $x,y$ satisfying the following properties.
	\begin{enumerate}
		\item \label{type:non-abelian} $[x,y] = y$, $\pthpower{x} = x$, $\pthpower{y} = 0$.
		\item \label{type:all zero}    $[x,y] = 0$, $\pthpower{x} = 0$, $\pthpower{y} = 0$.
		\item \label{type:1-0}         $[x,y] = 0$, $\pthpower{x} = x$, $\pthpower{y} = 0$.
		\item \label{type:nilpotent}   $[x,y] = 0$, $\pthpower{x} = y$, $\pthpower{y} = 0$.
		\item \label{type:1-1}         $[x,y] = 0$, $\pthpower{x} = x$, $\pthpower{y} = y$.
	\end{enumerate}
\end{prop}

We will use the following observations to describe the restricted Lie algebra of the canonical coverings.
\begin{cor} \label{cor:Lie algebra}
	Let $\fg$ be as in Proposition \ref{prop:classification of Lie algebra}.
	\begin{enumerate}
		\item \label{case:2-1}
		Suppose $\fg$ has at least $3$ $p$-closed lines, 
		among which at least $1$ is of multiplicative type and at least $1$ is of additive type.
		Then $\fg$ is of type (\ref{type:non-abelian}), 
		all lines are $p$-closed, 
		and exactly $1$ is of additive type and all others are of multiplicative type.
		
		\item \label{case:0-2}
		Suppose $\fg$ has at least $2$ lines of additive type.
		Then $\fg$ is of type (\ref{type:all zero}), 
		and all lines are $p$-closed of additive type.
	\end{enumerate}
\end{cor}
\begin{proof}
	Given the classification, we can describe the $p$-closed lines in each case by a straightforward calculation (see below).
	We conclude that if $\fg$ is of type (\ref{type:non-abelian}) or (\ref{type:all zero}) in Proposition \ref{prop:classification of Lie algebra}
	then the $p$-closed lines are as described in the statement of this corollary; 
	and if $\fg$ is of type (\ref{type:1-0}) (resp.\ (\ref{type:nilpotent}), resp.\ (\ref{type:1-1})),
	then exactly $1$ (resp.\ $0$, resp.\ $p + 1$) line is of multiplicative type,
	exactly $1$ (resp.\ $1$, resp.\ $0$) line is of additive type,
	and no other lines are $p$-closed.
	The assertions follow.
	
	For example, if $\fg$ is of type (\ref{type:non-abelian}) and $v = a x + b y$,
	then $\pthpower{v} = a^p x + a^{p-1} b y = a^{p-1} (a x + b y)$
	is always proportional to $v$, and $\pthpower{v} = 0$ if and only if $a = 0$.
	If $\fg$ is of type (\ref{type:1-1}) and $v = a x + b y \neq 0$,
	then $\pthpower{v} = a^p x + b^p y$ is never $0$,
	and it is proportional to $v$ 
	if and only if $\det{\begin{pmatrix}a^p & b^p \\ a & b \end{pmatrix}} = a b (a^{p-1} - b^{p-1}) = 0$.
\end{proof}

\section{$p$-closed derivations and quotients} \label{sec:derivation quotient}

\subsection{Derivations on RDPs and EDPs} \label{subsec:RDP EDP}

\begin{defn}
	An \emph{elliptic singularity} is an isolated surface singularity $x \in X$ 
	with $\length (R^1 f_* \cO)_x = 1$, where $f$ is a resolution of singularity.
	An \emph{elliptic double point} (EDP) is an elliptic singularity that is a double point.
\end{defn}

\begin{defn} \label{def:Ell12a}
	In this paper, 
	we say that a $2$-dimensional local $k$-algebra in characteristic $p = 2$ is an \emph{EDP of type $\Ella$}
	if its completion is isomorphic to $k[[x,y,z]] / (z^2 + x^3 + y^7)$.
\end{defn}
	This is the quotient of $k[[X,Y]]$ by the derivation $D$ defined by $D(X) = Y^6$ and $D(Y) = X^2$,
	with $x = X^2$, $y = Y^2$, $z = X^3 + Y^7$.

It is easy to see that it is an EDP
whose minimal resolution consists of a rational cuspidal curve of self-intersection $-1$. 
We observe that 
$k[[x,y,z]] / (z^2 + x^3 + y^7 + \varepsilon)$
is also an EDP of type $\Ella$
if $\varepsilon \in (x^5, x^3 y, x^2 y^3, x y^4, y^9) \subset k[[x,y]]$.

This symbol $\Ella$ is used for the (exceptional unimodal) singularity in characteristic $0$ defined by the same equation,
and the index $12$ stands for the Milnor number (i.e.\ $\dim_k k[[x,y,z]]/(F_x, F_y, F_z)$ for $k[[x,y,z]] / (F)$) in characteristic $0$,
although in characteristic $2$ this is not the Milnor number (nor the Tjurina number). 
Instead we have the equality between the index and the degree $\deg \isolatedfix{D}$ of the derivation.
The same equality also holds for RDPs of type $A_1$, $D_{2n}^0$, $E_7^0$, and $E_8^0$ (\cite{Matsumoto:k3alphap}*{Corollary 3.9}).

\begin{prop} \label{prop:image of derivation}
	Let $\bar{W} = \Spec B$ be an EDP of type $\Ella$ in characteristic $2$
	and $D$ a $p$-closed derivation on $\bar{W}$ with $\Fix(D) = \emptyset$.
	Then $Z = \bar{W}^D$ is smooth.
\end{prop}

\begin{proof}
	We may assume $B = k[[x,y,z]] / (z^2 + x^3 + y^7)$.
	The derivation $D$ satisfies $x^2 D(x) + y^6 D(y) = 0$, hence $D(x) = y^6 b$ and $D(y) = x^2 b$ for some $b \in B$.
	In particular $D(x)$ and $D(y)$ belong to the maximal ideal $\fm$ of $B$.
	Since $\Fix(D) = \emptyset$ we have $D(z) \in B^*$.
	Then the maximal ideal $\fn$ of $B^D$ is generated by three elements 
	\[ 
	x' := x - D(z)^{-1} D(x) z, \; y' := y - D(z)^{-1} D(y) z , \; z' = z^2,
	\]
	and since we have a relation 
	\begin{align*}
	z' = z^2 = x^3 + y^7 
	& = x^2 (x' + D(z)^{-1} D(x) z) + y^6 (y' + D(z)^{-1} D(y) z) \\
	&= x^2 x' + y^6 y'
	\in \fn^2,
	\end{align*}
	it is in fact generated by the two elements $x'$ and $y'$.
	Thus $\bar{W}^D$ is smooth.
\end{proof}

\begin{lem}[cf. \cite{Schroer:K3-like}*{Propositions 2.3--2.4}] \label{lem:tangent of RDP} 
	Suppose $B$ is the localization or the completion at a closed point of a normal surface in characteristic $p > 0$.
	Assume the closed point is a singularity with $\dim_k \fm / \fm^2 = 3$, where $\fm \subset B$ is the maximal ideal.
	Suppose $B$ admits a $p$-closed derivation $D$ with $\Fix(D) = \emptyset$.
	Then, 
	\begin{enumerate}
	\item \label{lem:tangent of RDP:free} 
	The tangent module $T_B = \Der(B)$ is a free $B$-module (of rank $2$).
	\item \label{lem:tangent of RDP:canonical line} 
	An element $D' \in T_B$ has no fixed points 
	if and only if the projection of $D'$ to $T_B \otimes B/\fm$ belongs to the complement of a certain line.
	\item \label{lem:tangent of RDP:2} 
	Assume $p = 2$. Suppose $D'_1, D'_2 \in T_B$ generate $T_B$ and that $\Fix(D'_1) \neq \emptyset$.
	If $B$ is an RDP of type $A_{2n-1}$ for some $n \geq 1$ (resp.\ any other singularity), 
	then $D'_1$ is not of additive type (resp.\ not of multiplicative type).
	\end{enumerate}
\end{lem}
	(\ref{lem:tangent of RDP:free}) and (\ref{lem:tangent of RDP:canonical line}) 
	slightly generalize
	the results of Schr\"oer \cite{Schroer:K3-like}*{Propositions 2.3--2.4}, in which $B$ is assumed to be of the form $k[[x,y,z]] / (z^p - f(x,y))$,
	and the proof is parallel.
	(\ref{lem:tangent of RDP:free}) also follows from \cite{Ekedahl--Hyland--Shepherd-Barron}*{Corollary 3.7(2)}.
	(\ref{lem:tangent of RDP:2}) generalizes the case of $A_1$ proved in \cite{Ekedahl--Hyland--Shepherd-Barron}*{Lemma 7.5}.
\begin{proof}
	By \cite{Matsumoto:k3alphap}*{Lemma 2.8}, 
	we can take $x,y,z \in \fm$ generating $\fm$ and satisfying $D(x) = D(y) = 0$.
	We may assume $B$ is complete.
	Hence we may assume 
	$B = k[[x,y,z]] / (F)$ with $F \in k[[x,y,z^p]]$.
	The tangent module $T_B$ can be identified with the $B$-module $\set{(a,b,c) \in B^3 \mid a F_x + b F_y + c F_z = 0}$
	by $D \mapsto (D(x), D(y), D(z))$.
	Here $F_x, F_y, F_z$ are the images in $B$ of the partial derivatives of $F$. 

	Since $F \in k[x,y,z^p]$ we have $F_z = 0$.
	Since the singularity is isolated (since $B$ is normal), the ideal $(F_x, F_y, F_z) = (F_x, F_y)$ is of height $2$. 
	Since $B$ is a hypersurface singularity, hence Cohen--Macaulay, this implies that $F_x, F_y$ is a regular sequence.
	Hence $T_B$ has a basis $D_1 = (0,0,1), D_2 = (F_y, - F_x, 0)$.
	This shows (\ref{lem:tangent of RDP:free}).
	Clearly $g_1 D_1 + g_2 D_2$ ($g_1, g_2 \in B$) has no fixed points if and only if $g_1 \in B^*$.
	This shows (\ref{lem:tangent of RDP:canonical line}).

	Now assume $p = 2$ and let $D'_1, D'_2$ be as in (\ref{lem:tangent of RDP:2}).
	We have $D'_1 = g_1 D_1 + g_2 D_2$ with $g_1 \in \fm$, $g_2 \in B^*$.
	Then we have $D'_1(x) = g_2 F_y \neq 0$ and 
	\begin{align*}
	 (D'_1)^2(x)  = D'_1 (g_2 F_y) 
	             &= (g_2 F_y)_x g_2 F_y + (g_2 F_y)_y g_2 F_x + (g_2 F_y)_z g_1
	          \\ &= \bigl( g_2 F_{yx} + (g_2)_x F_y + (g_2)_y F_x + (g_2)_z g_2^{-1} g_1 \bigr) D'_1(x).
	\end{align*}
	Note that $B$ is of type $A_{2n-1}$ for some $n \geq 1$ if and only if $F_{xy} \in B^*$
	(by \cite{Matsumoto:k3alphap}*{Theorem 3.3(1)}, $B$ cannot be of type $A_{2n}$).
	Assume this is the case (resp.\ not the case).
	Then since $F_x, F_y, g_1 \in \fm$ and $g_2 \in B^*$, the coefficient of $D'_1(x)$ is an element of $B^*$ (resp.\ $\fm$),
	in particular not equal to $0$ (resp.\ $1$).
\end{proof}

\subsection{Derivations on K3-like surfaces}

\begin{rem}
	The derivation corresponding to the canonical $\mu_2$- or $\alpha_2$-covering of a classical or supersingular Enriques surface is fixed-point-free.
	This follows from Bombieri--Mumford's construction \cite{Bombieri--Mumford:III}*{Corollary in Section 3}.
\end{rem}

We also have a partial converse:
\begin{prop}[cf.\ \cite{Matsumoto:k3alphap}*{Sections 3--4} and \cite{Schroer:K3-like}*{Proposition 5.1}] \label{prop:enriques derivation}
	Let $\bar{Y}$ be a normal K3-like surface in characteristic $2$ 
	with only RDPs or EDPs of type $\Ella$ as singularities.
	Let $D$ be a derivation of multiplicative (resp.\ additive) type satisfying $\Fix(D) = \emptyset$.
	Then,
	
	\begin{enumerate}
	\item \label{item:quotient}
	The quotient $\bar{X} := \bar{Y}^D$ is a classical (resp.\ supersingular) RDP Enriques surface,
	and $\bar{Y} \times_{\bar{X}} X$ is the canonical covering of the minimal resolution $X$ of $\bar{X}$.
	\item \label{item:quotient singularity}
	Let $\map{\bar{\pi}}{\bar{Y}}{\bar{X}}$ be the quotient map.
	If $w \in \bar{Y}$ is a closed point that is either a smooth point, an RDP of type $A_1$, $D_{2n}^0$, $E_7^0$, or $E_8^0$, or an EDP of type $\Ella$,
	then $\bar{\pi}(w)$ is smooth.
	If $w$ is an RDP of type $A_{2n-1}$ ($n \geq 2$), $D_{2n+1}^0$ ($n \geq 2$), or $E_6^0$,
	then $\bar{\pi}(w)$ is an RDP of type $A_{n-1}$, $A_1$, $A_2$ respectively.
	No other types of RDPs can appear on $\bar{Y}$.
	\item \label{item:smooth}
	If $\Sing(\bar{Y})$ has only RDPs, 
	then the total index of the RDPs on $\bar{Y}$ is $\geq 12$. and the equality holds if and only if $\bar{X}$ is a smooth Enriques surface.
	\end{enumerate}
\end{prop}
\begin{proof}
	(\ref{item:quotient singularity})
	This follows from \cite{Matsumoto:k3alphap}*{Theorem 3.3(1)} if $w$ is a smooth point or an RDP,
	and from Proposition \ref{prop:image of derivation} if $w$ is an EDP.

	(\ref{item:quotient})
	By (\ref{item:quotient singularity}), $\bar{X}$ has only RDPs as singularities (if any).
	Let $X \to \bar{X}$ be the minimal resolution.
	By \cite{Matsumoto:k3alphap}*{Theorem 3.3(1)},
	$\bar{Y} \times_{\bar{X}} X$ is also a normal K3-like surface with only RDPs and EDPs of type $\Ella$.
	Moreover $D$ extends to a regular derivation $\tilde{D} := D \otimes 1$ on $\bar{Y} \times_{\bar{X}} X$ of multiplicative (resp.\ additive) type 
	with $\Fix(\tilde{D}) = \emptyset$ and with quotient $(\bar{Y} \times_{\bar{X}} X)^{\tilde{D}} = X$.
	Hence we may assume $X = \bar{X}$ is smooth. 
	As in \cite{Matsumoto:k3alphap}*{proof of Proposition 4.1}, we have $K_{X} \equiv 0$,
	where $\equiv$ is the numerical equivalence.
	Hence, to show $X$ that is an Enriques surface, it suffices to show $\chi(\cO_X) = 1$.

Suppose $D$ is of multiplicative type. 
Then we have a decomposition $\bar{\pi}_* \cO_{\bar{Y}} = \bigoplus_{i \in \bZ/2\bZ} (\bar{\pi}_* \cO_{\bar{Y}})_i$
to eigenspaces of $D$ of eigenvalues $i \in \bZ/2\bZ$.
Since $\Fix(D) = \emptyset$, $(\bar{\pi}_* \cO_{\bar{Y}})_1 = \Image(D)$ is an invertible sheaf 
locally generated by an element of $\bar{\pi}_* \cO_{\bar{Y}}^* \cap (\bar{\pi}_* \cO_{\bar{Y}})_1$
and satisfies $((\bar{\pi}_* \cO_{\bar{Y}})_1)^{\otimes 2} \cong \cO_X$,
hence $(\bar{\pi}_* \cO_{\bar{Y}})_1$ is a $2$-torsion class in $\Pic(X)$.
In particular we have $\chi((\bar{\pi}_* \cO_{\bar{Y}})_1) = \chi((\bar{\pi}_* \cO_{\bar{Y}})_0) = \chi(\cO_X)$ by Riemann--Roch, 
hence $\chi(\cO_X) = \chi(\cO_{\bar{Y}}) / 2 = 1$ and $X$ is an Enriques surface.
If the class $(\bar{\pi}_* \cO_{\bar{Y}})_1 \in \Pic(X)$ is trivial, 
then $1 \in H^0(\bar{Y}, \cO_{\bar{Y}})$ would have a nontrivial square root
and $\bar{Y}$ would be non-reduced, which is absurd.
Therefore $\Pic(X)$ has nontrivial torsion, hence $X$ is classical,
and $\bar{Y}$ is the canonical covering of $X$.

Suppose $D$ is of additive type.
Since $\Fix(D) = \emptyset$, we have $\Image(D) = \cO_X$, 
and the extension
\[ 0 \to \cO_X \to \bar{\pi}_* \cO_{\bar{Y}} \namedto{D} \cO_X \to 0 \]
is non-split (otherwise $\bar{Y}$ would be non-reduced).
We obtain $\chi(\cO_X) = \chi(\cO_{\bar{Y}}) / 2 = 1$, hence $X$ is an Enriques surface.
Since $\bar{Y} \to X$ is purely inseparable, the Frobenius image $F(e)$ of the nontrivial class $e \in H^1(X, \cO_X)$ of this extension is zero.
This shows that $X$ is supersingular
and that $\bar{Y}$ is the canonical covering of $X$.
(cf.\ \cite{Bombieri--Mumford:III}*{Corollary in Section 3}.)

	(\ref{item:smooth})
	Assume $\Sing(\bar{Y})$ has only RDPs.
	Let $n_i$ and $m_j$ be the indices of the RDPs on $\bar{Y}$ and $\bar{X}$ respectively.
	Then we have $b_2(\bar{Y}) = b_2(Y) - \sum n_i = 22 - \sum n_i$ and
	$b_2(\bar{X}) = b_2(X) - \sum m_j = 10 - \sum m_j$,
	where $Y \to \bar{Y}$ and $X \to \bar{X}$ are the minimal resolutions.
	Since $\bar{\pi}$ is purely inseparable we have $b_2(\bar{Y}) = b_2(\bar{X})$.
	Hence $\sum n_i = 12 + \sum m_j \geq 12$ and the equality is equivalent to $\sum m_j = 0$. 
\end{proof}

We slightly generalize the results of Ekedahl--Hyland--Shepherd-Barron and Schr\"oer 
on the tangent sheaf of the canonical covering and the fixed loci of global sections.

\begin{prop}[cf.\ \cite{Ekedahl--Hyland--Shepherd-Barron}*{Corollary 3.7(3)}, \cite{Schroer:K3-like}*{Theorem 6.4}] \label{prop:T free}
	Suppose $\bar{Y}$ and $D$ are as in Proposition \ref{prop:enriques derivation}.
	Then,
	\begin{enumerate}
	\item \label{prop:T free:T free} 
	The tangent sheaf $T_{\bar{Y}}$ is free (of rank $2$).
	\item \label{prop:T free:canonical line} 
	For each $w \in \Sing(\bar{Y})$ there exists a line $\fl(w) \subset H^0(\bar{Y}, T_{\bar{Y}})$ 
	such that, for $D' \in H^0(\bar{Y}, T_{\bar{Y}})$,
	we have $w \in \Fix(D')$ if and only if $D' \in \fl(w)$.
	\item \label{prop:T free:canonical lines} 
	An element $D' \in H^0(\bar{Y}, T_{\bar{Y}})$ is fixed-point-free 
	if and only if it belongs to the complement of the (finite) union of the lines $\fl(w)$. 
	\end{enumerate}
\end{prop}
	Again, if we assume moreover $X = Y^D$ is smooth,
	then under some assumption on $\bar{Y}$ 
	the assertions follow from \cite{Schroer:K3-like}*{Proposition 6.1 and Theorem 6.4},
	and the proofs of  
	(\ref{prop:T free:canonical line}) and 
	(\ref{prop:T free:canonical lines})
	are parallel.

\begin{proof}
(\ref{prop:T free:T free})
$T_{\bar{Y}}$ is locally free by Lemma \ref{lem:tangent of RDP}(\ref{lem:tangent of RDP:free}).
Then we can apply the proof of \cite{Ekedahl--Hyland--Shepherd-Barron}*{Corollary 3.7(3)} as follows
(although it is stated for smooth Enriques surfaces). 
Since $D$ is fixed-point-free,
the quotient $L := T_{\bar{Y}} / \cO_{\bar{Y}} D$ is an invertible sheaf.
Since $K_{\bar{Y}^{\sm}} = 0$, comparing the Chern classes we obtain $L \cong \cO_{\bar{Y}}$.
Since $H^1(\bar{Y}, \cO) = 0$, the extension is trivial.

(\ref{prop:T free:canonical line}),
(\ref{prop:T free:canonical lines})
We can apply the proof of \cite{Schroer:K3-like}*{Theorem 6.4} as follows
(although it is stated for smooth Enriques surfaces). 
For each closed point $w \in \bar{Y}$, the composite 
$H^0(\bar{Y}, T_{\bar{Y}}) \to T_{\bar{Y},w} \to T_{\bar{Y},w} \otimes k(w)$ is an isomorphism of restricted Lie algebras.
If $w$ is a smooth point, then $w \in \Fix(D')$ if and only if $D' = 0$.
If $w$ is a singular point, 
then $w \in \Fix(D')$ if and only if $D' \in \fl(w)$,
where $\fl(w) \subset H^0(\bar{Y}, T_{\bar{Y}})$ is the inverse image 
of the line of $T_{\bar{Y},w} \otimes k(w)$
mentioned in Lemma \ref{lem:tangent of RDP}(\ref{lem:tangent of RDP:canonical line}).
Therefore $\Fix(D') = \emptyset$ if and only if $D' \not\in \set{0} \cup \bigcup_{w \in \Sing(\bar{Y})} \fl(w)$.
Since $\bar{Y}$ has at least one singular point (by Proposition \ref{prop:enriques derivation}(\ref{item:smooth})), we have $0 \in \bigcup \fl(w)$.
\end{proof}

Following Schr\"oer \cite{Schroer:K3-like}*{Section 2}, 
we call this line $\fl(w) \subset \fg = H^0(\bar{Y}, T_{\bar{Y}})$ to be the \emph{canonical line} attached to $w \in \Sing(\bar{Y})$.

\begin{cor} \label{cor:canonical line}
	Suppose $\bar{Y}$ and $D$ are as in Proposition \ref{prop:enriques derivation}.
	If $w \in \Sing(\bar{Y})$ is an RDP of type $A_{2n-1}$ for some $n \geq 1$ (resp.\ any other singularity),
	then the attached canonical line $\fl(w) \subset \fg$ is of multiplicative type (resp.\ of additive type).
\end{cor}
\begin{proof}
The canonical line $\fl(w)$ is $p$-closed since $\Fix(D^p) \supset \Fix(D) \ni w$ for $D \in \fl(w)$. 
Hence it is either of multiplicative type or of additive type.
Take a generator $D'_1$ of $\fl(w)$ and extend it to a basis $D'_1, D'_2$ of $\fg$.
Then Lemma \ref{lem:tangent of RDP}(\ref{lem:tangent of RDP:2}) excludes one possibility.
\end{proof}

\section{Proof of the main theorem} \label{sec:proof}

Hereafter we assume $p = 2$.

In this section we prove Theorem \ref{thm:main3}.
Our proof is case-by-case:
$\bar{Y}$ having an EDP are discussed in Section \ref{subsec:proof:EDP},
those with only RDPs of type $D_n$ and $E_n$ (those with $\Sing(\bar{Y}) = 3 D_4^0$, $D_4^0 + D_8^0$, $D_4^0 + E_8^0$, or $D_{12}^0$) in Section \ref{subsec:proof:DE},
and those having at least one RDP of type $A_1$ (those with $\Sing(\bar{Y}) = 12 A_1$, $8 A_1 + D_4^0$, $6 A_1 + D_6^0$, or $5 A_1 + E_7^0$) in Section \ref{subsec:proof:A}.
By Theorems \ref{thm:classification of canonical covering} and \ref{thm:EHSB},
this covers all cases to be considered.

Before splitting into cases, we note the following.
	\begin{prop}[cf.\ \cite{Matsumoto:k3alphap}*{Theorem 9.1}] \label{prop:nexists:ss}
		Let $\bar{Y}$ be a normal surface that is the canonical covering of a supersingular Enriques surface $X$.
		Then $\Sing(\bar{Y})$ is one of 
		\[ 12 A_1, \, 3 D_4^0, \, D_4^0 + D_8^0, \, D_4^0 + E_8^0, \, D_{12}^0, \, \text{or } \Ella. \]
	\end{prop}
	\begin{proof}
		We have $K_X = 0$ since $X$ is supersingular, and $K_{\bar{Y}} = 0$ and $\bar{Y}$ is normal by assumption.
		Then the ``dual'' morphism $\map{\pi'}{\thpower{X}{1/2}}{\bar{Y}}$ is,
		by the argument in \cite{Matsumoto:k3alphap}*{proof of Theorem 4.3},
		the quotient morphism by 
		either a $\mu_2$- or $\alpha_2$-action 
		with only isolated fixed points,
		and the fixed locus $\isolatedfix{D'}$ of the corresponding derivation $D'$ on $\thpower{X}{1/2}$ has degree $12$ 
		by the Katsura--Takeda formula. 
		We use the classification (\cite{Matsumoto:k3alphap}*{Lemma 3.6 and Corollary 3.9}) of $\mu_2$- and $\alpha_2$-quotient singularities with degree $\leq 12$.
		If it is a $\mu_2$-quotient, then each singular point of $\bar{Y}$ is an RDP of $A_1$.
		If it is an $\alpha_2$-quotient, then each singular point of $\bar{Y}$ is an RDP of type $D_{4n}^0$ or $E_8^0$ or an EDP of type $\Ella$.
		In each case, the degree of $\isolatedfix{D'}$ at each point is equal to the index of the quotient singularity.
	\end{proof}

\subsection{Case of \texorpdfstring{$\bar{Y}$}{Y} with an EDP} \label{subsec:proof:EDP}

\begin{rem} \label{rem:CD131}
	Suppose $\bar{Y}$ has a non-RDP.
	It is claimed in the proof of \cite{Cossec--Dolgachev:enriques}*{Proposition 1.3.1} that 
	then $\bar{Y}$ has exactly one non-RDP singularity and it is an EDP.
	The proof is however incomplete where they use the Leray spectral sequence.
	This is fixed in the new version of the book \cite{Cossec--Dolgachev--Liedtke:enriques1}.
	Schr\"oer \cite{Schroer:K3-like}*{proof of Proposition 5.4} also gives an argument.
	We can also use 
	the classification (\cite{Matsumoto:k3alphap}*{Lemma 3.6 and Corollary 3.9}) of $2$-closed derivation quotient singularities with small degree,
	saying that the singularity is an RDP if degree $\leq 10$
	and that the singularity is either an RDP or an EDP if degree $\leq 12$.
\end{rem}

The essential part of the proof of this case is:
\begin{prop} \label{prop:nexists:cl:EDP}
	Suppose $X$ is a classical Enriques surface
	whose canonical covering $\bar{Y}$ is normal.
	Then $\Sing(\bar{Y})$ does not contain an EDP.
\end{prop}

\begin{defn} \label{def:radical2section}
Following Schr\"oer \cite{Schroer:K3-like}*{Section 8}, we say that an integral curve $A \subset X$
is a \emph{radical two-section} of an elliptic fibration $X \to \bP^1$
if the composite $A \to X \to \bP^1$ is surjective and inseparable of degree $2$.
\end{defn}
Following arguments of \cite{Schroer:K3-like}*{proof of Proposition 8.9},
we can prove the following assertion
on Enriques surfaces having no elliptic fibrations admitting a radical two-section.

\begin{lem}[cf.\ \cite{Schroer:K3-like}*{Proposition 8.9}] \label{lem:radical2section}
	Suppose $X$ is a classical or supersingular Enriques surface
	whose canonical covering $\bar{Y}$ is normal.
	Assume that no elliptic fibration on $X$ admits a radical two-section.
	Then either $X$ is supersingular or $\# \Sing(\bar{Y}) \geq 5$.
\end{lem}
\begin{proof}
	Since $\bar{Y}$ is normal,
	any genus one fibration on $X$ is elliptic (\cite{Schroer:K3-like}*{Theorem 5.6(i)}).
	Suppose no elliptic fibration on $X$ admits a radical two-section.
	Then $X$ does not admit a smooth rational curve 
	nor a non-movable cuspidal rational curve (\cite{Schroer:K3-like}*{Proposition 8.8})
	nor a non-movable nodal rational curve (same proof as in the cuspidal case).
	Let $\map{\phi}{X}{\bP^1}$ be an elliptic fibration
	and $\map{\phi'}{J}{\bP^1}$ its Jacobian fibration.
	By above, any half-fiber of $\phi$ is smooth,
	and any singular fiber of $\phi$ is of Kodaira type $\rI_1$ or $\rII$.
	(We call $(X_a)_{\red}$ a half-fiber if $X_a$ is a multiple fiber of multiplicity $2$.)
	By \cite{Liu--Lorenzini--Raynaud:neron}*{Theorem 6.6},
	if a fiber of an elliptic fibration is of type $m T$, 
	where $m \in \bZ_{>0}$ is the multiplicity and $T \in \set{\rI_n, \rI_n^*, \rII, \rII^*, \rIII, \rIII^*, \rIV, \rIV^*}$ is the symbol denoting the Kodaira type,
	then the corresponding fiber of its Jacobian fibration is of type $T$. 
	Hence $\phi'$ has the same types of singular fibers as $\phi$ (up to multiplicity).

	Suppose $\phi'$ has no fibers of type $\rI_1$.
	Then, 
	by Lang's classification of configurations of singular fibers of rational elliptic surfaces (\cite{Lang:configurations}*{Section 2 or 4}), 
	the relative $j$-invariant for $\phi'$ is $0$.
	This shows that any smooth fiber of $\phi'$ is a supersingular elliptic curve.
	Let $(X_a)_{\red}$ be a half-fiber of $\phi$. 
	Then it is smooth by above,
	and isogenous to the corresponding fiber $J_a$ of $\phi'$
	(consider the base change to a finite cover $C \to \bP^1$ over which $\phi$ acquires a section),
	hence supersingular.
	Then $X$ cannot be classical by Proposition \ref{prop:multiple fibers}.

	Now suppose there is at least one fiber of type $\rI_1$ (and no singular fiber of type other than $\rI_1$ and $\rII$).
	Again by Lang's classification (\cite{Lang:configurations}*{Sections 2--3 or 4}), we observe that 
	the singular fibers of $\phi'$, and hence those of $\phi$, are $12 \rI_1$, $8 \rI_1 + \rII$, $6 \rI_1 + \rII$, or $5 \rI_1 + \rII$.
	By \cite{Schroer:K3-like}*{Proposition 4.7},
	the point above the node of each fiber of type $\rI_1$ is a singular point of $\bar{Y}$.
	Hence $\bar{Y}$ has at least $5$ singular points.
\end{proof}

\begin{proof}[Proof of Proposition \ref{prop:nexists:cl:EDP}]
	Since $\bar{Y}$ is normal, the ``dual'' morphism $X \to \secondpower{\bar{Y}}$ is the quotient by 
	a rational derivation $D'$ on $X$.
	We have $\Sing(\bar{Y}) = \bar{\pi}^{-1}(\Supp \isolatedfix{D'})$.
	By the Rudakov--Shafarevich formula we have $\divisorialfix{D'} \sim - K_X \equiv 0$, 
	hence by the Katsura--Takeda formula we have $\deg \isolatedfix{D'} = 12$.
	Here $\equiv$ is the numerical equivalence.
	By \cite{Matsumoto:k3alphap}*{Corollary 3.9}, if the quotient singularity on $\secondpower{\bar{Y}}$ is an EDP then $\isolatedfix{D'}$ has degree at least $11$ at the corresponding point of $X$.
	Hence if $X$ has an EDP then $\# \Sing(\bar{Y}) \leq 2$.
 
	Since $X$ is classical,
	we may assume by Lemma \ref{lem:radical2section}
	that $X$ admits an elliptic fibration $\map{\phi}{X}{\bP^1}$ with a radical two-section.
	Then by \cite{Schroer:K3-like}*{Propositions 8.1 and 8.5},
	$\map{\phi \circ \bar{\pi}}{\bar{Y}}{\bP^1}$ factors as $\bar{Y} \namedto{\psi} \bP^1 \namedto{F} \bP^1$
	and this $\psi$ admits a section (e.g.\ $\bar{\pi}^{-1}(A)_{\red}$ for any radical two-section $A$ of $\phi$).
	Let $\map{\phi'}{J}{\bP^1}$ be the Jacobian fibration of $\phi$,
	and $\map{\psi'}{\thpower{J}{2/\bP^1} := J \times_{\bP^1} \bP^1}{\bP^1}$ be the Frobenius base change of $\phi'$.
	Then the existence of a section of $\psi$ implies that 
	the generic fiber of $\map{\psi}{\bar{Y}}{\bP^1}$ is 
	isomorphic to the generic fiber of $\map{\psi'}{\thpower{J}{2/\bP^1}}{\bP^1}$
	by \cite{Schroer:K3-like}*{Proposition 8.4}.
	In particular $\bar{Y}$ and $\thpower{J}{2/\bP^1}$ are birational.
	As above, if $X_a$ is of type $m T$ then $J_a$ is of type $T$.
	Since $\bar{Y}$ is normal, we have $T \in \set{\rI_n, \rII, \rIII, \rIV}$ by \cite{Schroer:K3-like}*{Theorem 5.6(ii)},
	in particular $J_a$ is reduced for all $a \in \bP^1$.
	By \cite{Schroer:K3-like}*{Proposition 11.1}, $\thpower{J}{2/\bP^1}$ also has trivial dualizing sheaf.
	By \cite{Schroer:K3-like}*{Proposition 11.2}, $\Sing(\thpower{J}{2/\bP^1})$ is precisely the points over the non-smooth locus of $J \to \bP^1$,
	and then it is isolated since $J$ has only finitely many singular fibers and all of them are reduced. 

	Suppose $\bar{Y}$ has an EDP.
	Then $\bar{Y}$ and hence $\thpower{J}{2/\bP^1}$ are rational surfaces.
	Since $\thpower{J}{2/\bP^1}$ has trivial dualizing sheaf and $\Sing(\thpower{J}{2/\bP^1})$ is isolated, 
	$\thpower{J}{2/\bP^1}$ also has a non-RDP singularity.
	By \cite{Schroer:K3-like}*{Theorem 12.1}, based on Lang's classification \cite{Lang:extremalII}*{Section 2A} of local Weierstrass equations in characteristic $2$, 
	this can happen only if the corresponding fiber of $J$ is of Lang type 9C (i.e. 
	$J$ is of the form
	\[
	y^2 + t^3 \gamma_0 y = x^3 + t \gamma_1 x^2 + t \gamma_3 x + t \gamma_5 ,
	\]
	with polynomials $\gamma_i \in k[t]$ of degree $\leq i$ satisfying $t \notdivides \gamma_0$ and $t \notdivides \gamma_5$)
	and moreover $t \divides \gamma_3$. 
	In particular, $\map{\phi'}{J}{\bP^1}$ has only one singular fiber (at $t = 0$)
	and all remaining fibers are supersingular elliptic curves.

	As in the previous lemma, $(X_a)_{\red}$ is smooth if and only if $J_a$ is smooth,
	and in this case these elliptic curves are isogenous. 
	Hence $\map{\phi}{X}{\bP^1}$ has, up to multiplicity, only one singular fiber and all remaining fibers are supersingular elliptic curves.

	On the other hand, since $X$ is classical, the elliptic fibration $\map{\phi}{X}{\bP^1}$ has two multiple fibers,
	and each multiple fiber is either a smooth ordinary elliptic curve or a singular fiber of additive type 
	(Proposition \ref{prop:multiple fibers}). 
	Contradiction.
\end{proof}

	The supersingular case remains.

\begin{proof}[Proof of Theorem \ref{thm:main3} in the case $\bar{Y}$ has an EDP]
Let $D$ be a fixed-point-free derivation on $\bar{Y}$ with Enriques quotient $X := Y^D$.
By Proposition \ref{prop:nexists:cl:EDP}, $X$ is supersingular.
By Proposition \ref{prop:nexists:ss}, $\Sing(\bar{Y})$ consists of one point, of type $\Ella$.
By Corollary \ref{cor:canonical line}, the canonical line $\fl$ attached to the singularity is of additive type.
Since the $2$ lines $[D]$ and $\fl$ of $\fg$ of additive type are distinct (Proposition \ref{prop:T free}(\ref{prop:T free:canonical lines})), 
it follows from Corollary \ref{cor:Lie algebra}(\ref{case:0-2}) that 
all lines of $\fg$ are of additive type
and that $\fg$ is abelian.
\end{proof}

\begin{rem}
	Combining Propositions \ref{prop:nexists:cl:EDP} and \ref{prop:nexists:ss}, 
	we obtain another proof of Schr\"oer's result \cite{Schroer:K3-like}*{Theorem 14.1}
	that if $\bar{Y}$ has an EDP then it has no other singularities.
\end{rem}

\subsection{Case of \texorpdfstring{$\bar{Y}$}{Y} with only RDPs of type $D_n$ or $E_n$} \label{subsec:proof:DE}

The following lemma on RDP K3 surfaces follows from arguments in \cite{Matsumoto:k3alphap}.

\begin{lem} \label{lem:cohomology of RDP}
	Suppose $\bar{Y}$ is an RDP K3 surface with $\Sing(\bar{Y}) \neq \emptyset$,
	with $\Sing(\bar{Y}) = \set{w_i}_{i = 1}^N$,
	and $(n_i)_{i = 1}^N$ are positive integers 
	such that for each $i$ one of the following holds.
	\begin{itemize}
		\item $w_i$ is an RDP of type $D_{4n_i}^0$.
		\item $w_i$ is an RDP of type $E_8^0$ and $n_i = 2$.
	\end{itemize} 
For each $i$, 
let $I_{w_i} \subset \cO_{\bar{Y},{w_i}}$ be the ideal defined in \cite{Matsumoto:k3alphap}*{Section 6.2},
and let $\cI = \Ker(\cO_{\bar{Y}} \to \bigoplus_i \cO_{\bar{Y},{w_i}} / I_{w_i})$.
	Then,
	\begin{enumerate}
	\item \label{lem:cohomology of RDP:dim}
	the Frobenius map $\map{F}{\Ext^1_{\bar{Y}}(\cI, \cO)}{\Ext^1_{\bar{Y}}(\secondpower{\cI}, \cO)}$ is zero
	and we have $\dim \Ext^1_{\bar{Y}}(\cI, \cO) = -1 + \sum n_i$.
	\item \label{lem:cohomology of RDP:derivation}
	There is a family $(\bar{Z}'_e, D_e)$
	of $\alpha_2$-coverings $\map{\pi'_e}{\bar{Z}'_{e}}{\bar{Y}^{\sm}}$
	and global derivations $D_e \in H^0(\bar{Y}, T_{\bar{Y}})$ of additive type,
	parametrized by $e \in \Ext^1_{\bar{Y}}(\cI, \cO)$,
	such that 
	\begin{itemize}
	\item $\Sing(\bar{Z}'_e) = \pi'_e(\Fix(D_e \restrictedto{\bar{Y}^{\sm}}))$,
	\item The sequence  
	$0 \to \cO_{\bar{Y}^{\sm}} \to \cO_{\bar{Z}'_e} \namedto{\delta} \cO_{\bar{Y}^{\sm}} \to 0$,
	where $\delta$ is the derivation corresponding to the $\alpha_2$-action,
	is exact and represents the restriction of $e$ to $\bar{Y}^{\sm}$, and
	\item $\Ext^1_{\bar{Y}}(\cI, \cO) \to H^0(\bar{Y}, T_{\bar{Y}}) \colon e \mapsto D_e$ 
	is an injective semilinear map.
	\end{itemize}
	\end{enumerate}
\end{lem}
\begin{proof}
	(\ref{lem:cohomology of RDP:dim})
	Consider the commutative diagram with exact rows
		\[
		\begin{tikzcd}
		0                                    \arrow[r] &
		\Ext^1_{\bar{Y}}(\cI,     \cO)               \arrow[r] \arrow[d,"F"] &
		\bigoplus_{i} \Ext^2_{\bar{W}_i}(\cO/I_{w_i},     \cO) \arrow[r] \arrow[d,"F"] &
		H^2(Y, \cO)                          \arrow[r] \arrow[d,"F"] &
		0 \\
		0                                    \arrow[r] &
		\Ext^1_{\bar{Y}}(\secondpower{\cI},     \cO)               \arrow[r] &
		\bigoplus_{i} \Ext^2_{\bar{W}_i}(\cO/\secondpower{I_{w_i}},     \cO) \arrow[r] &
		H^2(Y, \cO)                          \arrow[r] &
		0 
		\end{tikzcd}
		\]
	constructed as in \cite{Matsumoto:k3alphap}*{proof of Theorem 7.3(2)},
	where, for each $i$, $\bar{W}_i = \Spec \hat{\cO}_{\bar{Y},w_i}$ is the completion at the RDP $w_i$ of $\bar{Y}$.
	As proved in \cite{Matsumoto:k3alphap}*{Lemma 6.6(1)}, 
	the Frobenius map $\map{F}{\Ext^2_{\bar{W}_i}(\cO/I_{w_i}, \cO)}{\Ext^2_{\bar{W}_i}(\cO/\secondpower{I_{w_i}}, \cO)}$ associated with the local ring $\bar{W}_i$ are zero.
	This implies the former assertion.
	
	The latter equality follows from 
	$\dim_k \Ext^2_{\bar{W}_i}(\cO/I_{w_i}, \cO) = \dim_k (\cO/I_{w_i}) = n_i$
	and $\dim H^2(Y, \cO) = 1$.

(\ref{lem:cohomology of RDP:derivation})
	(This construction imitates Bombieri--Mumford's construction \cite{Bombieri--Mumford:III}*{Section 3} 
	of the canonical $\alpha_2$-covering of a supersingular Enriques surface $X$
	from a nontrivial class in $H^1(X, \cO)^{F = 0} = H^1(X, \cO)$.)
	
	We fix a nontrivial $2$-form $\omega$ on $Y^{\sm}$.
	Take a class $e \in \Ext^1_{\bar{Y}}(\cI, \cO)$ and consider the corresponding extension 
	\[
		0 \to \cO \to V \namedto{\delta} \cI \to 0.
	\] 
	Then we obtain, as in \cite{Matsumoto:k3alphap}*{proof of Theorem 7.3(2)},
	an $\alpha_2$-covering $\map{\pi'_e}{\bar{Z}'_e}{\bar{Y}^{\sm}}$
	with $V \restrictedto{\bar{Y}^{\sm}} = \cO_{\bar{Z}'_e}$
	and $\delta$ being the derivation corresponding to the $\alpha_2$-action.
	As in \cite{Matsumoto:k3alphap}*{Proposition 2.15} we define 
	a $1$-form $\eta$ on $\bar{Y}^{\sm}$
	and a $p$-closed derivation $D$ on $\bar{Y}^{\sm}$
	in the following way:
	let $t$ a local section of $V$ such that $\delta(t) = 1$,
	so that $\bar{Z}'_e$ is locally defined as 
	$\cO_{\bar{Z}'_e} = \cO_{\bar{Y}^{\sm}}[t] / (t^2 - c)$ with $c \in \cO_{\bar{Y}^{\sm}}$,  
	let $\eta = d c$, and define $D$ by $D(f) \omega = df \wedge \eta$.
	Then we have $\pi'_e(\Sing(\bar{Z}'_e)) = \Fix(D) = \Zero(\eta)$.
	
	Since $\bar{Y}$ is normal, the derivation $D$ on $\bar{Y}^{\sm}$ extends to one on $\bar{Y}$.
	This map $\Ext^1_{\bar{Y}}(\cI, \cO) \to H^0(\bar{Y}, T_{\bar{Y}}) \colon e \mapsto D$ is $F$-semilinear by construction. 
	We will show that this is injective.
	Suppose $D = 0$. Then $\eta = 0$.
	Then $c = b^2$ for some local sections $b$ of $\cO_{\bar{Y}^{\sm}}$.
	Then $t' := t - b$ glue to a global section $t' \in H^0(\bar{Y}^{\sm}, V)$ with $\delta(t') = 1$,
	and moreover to a global section on $\bar{Y}$,
	hence the extension is trivial and $e = 0$.
\end{proof}

\begin{proof}[Proof of Theorem \ref{thm:main3} in the case $\Sing(\bar{Y})$ is $3 D_4^0$, $D_4^0 + D_8^0$, $D_4^0 + E_8^0$, or $D_{12}^0$]
	By Lemma \ref{lem:cohomology of RDP}
	we obtain a $2$-dimensional family $\bar{Z}'_e$ of $\alpha_2$-coverings of $\bar{Y}^{\sm}$
	parametrized by $e \in \Ext^1_{\bar{Y}}(\cI, \cO) \cong H^0(\bar{Y}, T_{\bar{Y}})$.
	We show that if $e \neq 0$ then this extends to a family $\bar{Z}_e$ of $\alpha_2$-coverings of $\bar{Y}$,
	and show that the family $(\secondpower{\bar{Z}_{e}})_{e \in \Ext^1_{\bar{Y}}(\cI, \cO) \setminus \set{0} }$ 
	exhaust nontrivial $p$-closed derivation quotients of $Y$.

	Suppose $e \neq 0$ and let $D \neq 0$ be the corresponding derivation on $\bar{Y}$.
	Since $T_{\bar{Y}}$ is free, $D$ has no fixed points on $\bar{Y}^{\sm}$.
	Then $\bar{Z}'_e$ is normal,
	since it is regular outside the codimension $2$ subscheme ${\pi'_e}^{-1}(\Fix(D))$
	and is Gorenstein everywhere.
	Let $\bar{Z}_e \to \bar{Y}$ be the normalization of $\bar{Y}$ in $k(\bar{Z}'_e)$. 
	Then the derivation $\delta$ on $\bar{Z}'_e$ extends to a derivation of $\bar{Z}_e$,
	which defines an $\alpha_2$-action with quotient $\bar{Y}$.
	Since $\bar{Z}_e$ is normal and 
	$\cO_{\secondpower{\bar{Z}_e}} \restrictedto{\bar{Y}^{\sm}} = V^{(2)} \restrictedto{\bar{Y}^{\sm}} \subset \Ker D$,
	we obtain $\bar{Y}^D = \secondpower{\bar{Z}_{e}}$.
	Since $[k(\bar{Y}) : k(\secondpower{\bar{Z}_e})] = 2$, it follows that $D$ is $p$-closed, 
	and since $T_{\bar{Y}}$ is free, we have $D^2 = \lambda D$ with $\lambda \in k$.
	Note that replacing $e$ with a nonzero multiple replaces $D$ with a nonzero multiple, hence results in the same quotient.

	Take any $D$ that does not belong to any canonical line.
	Then $D$ is fixed-point-free by Proposition \ref{prop:T free}(\ref{prop:T free:canonical lines}), 
	hence $\bar{Y}^D = \secondpower{\bar{Z}_e}$ is an Enriques surface.
	It is supersingular and thus $D^2 = 0$,
	since a classical Enriques surface does not admit a regular $p$-closed derivation with K3-like quotient 
	by the Rudakov--Shafarevich formula (cf.\ \cite{Matsumoto:k3alphap}*{Proposition 4.5}).
	By Corollary \ref{cor:Lie algebra}(\ref{case:0-2}), $\fg$ is abelian and all elements $D$ satisfy $D^2 = 0$.
	In particular $\bar{Y}$ has no $p$-closed derivation quotient that is a classical Enriques surface.
\end{proof}

\subsection{Case of \texorpdfstring{$\bar{Y}$}{Y} having $A_1$} \label{subsec:proof:A}

\begin{proof}[Proof of Theorem \ref{thm:main3} in the case $\Sing(\bar{Y})$ is $8 A_1 + D_4^0$, $6 A_1 + D_6^0$, or $5 A_1 + E_7^0$]
Let $w_1$ be a singular point of type $A_1$ 
and $w_2$ a singular point not of type $A_1$.
Then the attached canonical lines $\fl(w_1)$ and $\fl(w_2)$ of $\fg = H^0(\bar{Y}, T_{\bar{Y}})$ are respectively of multiplicative type and additive type 
by Corollary \ref{cor:canonical line}.
The line generated by a fixed-point-free $p$-closed derivation (which exists by assumption) 
is different from $\fl(w_1)$ and $\fl(w_2)$ (Proposition \ref{prop:T free}(\ref{prop:T free:canonical lines})).
By Corollary \ref{cor:Lie algebra}(\ref{case:2-1}),
all lines of $\fg$ are $p$-closed, 
and among them exactly one is of additive type, which should be $\fl(w_2)$.
Hence all Enriques quotients of $\bar{Y}$ are classical.
The assertion on the bracket follows from Corollary \ref{cor:Lie algebra}(\ref{case:2-1}).
\end{proof}

\begin{proof}[Proof of Theorem \ref{thm:main3} in the case $\Sing(\bar{Y})$ is $12 A_1$]
If all $12$ canonical lines are equal, 
then a generator of the line extends to a derivation on the blow-up $Y$ of $\bar{Y}$ at the $12$ points,
but since $Y$ is a (smooth) K3 surface this is impossible by \cite{Rudakov--Shafarevich:inseparable}*{Theorem 7}.
Hence there are at least $2$ distinct canonical lines, both of multiplicative type by Corollary \ref{cor:canonical line}.

By applying Proposition \ref{prop:half fix} to the rational derivation on $Y$ induced by a fixed-point-free derivation $D$, 
where $Y \to \bar{Y}$ is the minimal resolution with exceptional curves $\set{e_w}_{w \in \Sing(\bar{Y})}$,
we see that 
$\sum_{w \in \Sing(\bar{Y})} e_w \in 2 \Pic(Y)$.
This induces, as in \cite{Matsumoto:k3alphap}*{Theorem 7.3(2)}, 
a $\mu_2$-covering $\bar{Z} \to \bar{Y}$ that is regular above a neighborhood of $\Sing(\bar{Y})$.
Let $D' \neq 0$ be the resulting $p$-closed derivation on $\bar{Y}$ 
(cf. \cite{Matsumoto:k3alphap}*{proof of Theorem 7.3(2)}).
Then $D'$ is fixed-point-free, since $\Fix(D')$ contains none of $\Sing(\bar{Y})$. 
Hence $\bar{Z}$ is an Enriques surface.
As in the previous subsection, it is supersingular.
Hence the line $[D'] \subset \fg$ is of additive type.

By Corollary \ref{cor:Lie algebra}(\ref{case:2-1}),
all lines of $\fg = H^0(\bar{Y}, T_{\bar{Y}})$ are $p$-closed and among them exactly one is of additive type, 
which is $[D']$, which is fixed-point-free. 
\end{proof}

\section{Examples} \label{sec:examples}

In this section we prove Theorem \ref{thm:main4}.
If $\bar{Y}$ contains a non-RDP singularity, then $\bar{Y}$ has an EDP by Theorem \ref{thm:classification of canonical covering}, and we proved in Theorem \ref{thm:main3} that $\bar{Y}$ has one EDP of type $\Ella$ and contains no other singularity.
If $\Sing(\bar{Y})$ consists only of RDPs, then the configuration is one of the $8$ given in Theorem \ref{thm:EHSB}.
Hence it remains to show that each of the $9$ configuration is indeed possible.
We will give explicit examples.

\subsection{Examples of canonical coverings that are RDP K3 surfaces} \label{sec:examples:RDP} 

It turns out that all configurations of RDPs 
are realized by 
Enriques surfaces admitting elliptic fibrations admitting a radical two-section (Definition \ref{def:radical2section}).

In each example, we give two elliptic RDP K3 surfaces $\bar{Y}' \to \bP^1$ and $\bar{Y}'' \to \bP^1$ satisfying the following properties.
\begin{itemize}
\item The generic fibers of $\bar{Y}'$ and $\bar{Y}''$ are isomorphic.
\item $\bar{Y}'$ is isomorphic to the Frobenius base change $\bar{J} \times_{\bP^1} {\bP^1}$ 
of the Weierstrass form $\bar{J} \to \bP^1$ of some rational elliptic surface $J \to \bP^1$.
\item We give a basis $D_1, D_2$ for $H^0(\bar{Y}'', T_{\bar{Y}''})$. 
A generic element $D = e_1 D_1 + e_2 D_2$ ($e_1, e_2 \in k$) has no fixed points,
hence the quotient $\bar{X} := (\bar{Y}'')^D$ is an RDP Enriques surface,
and $\bar{Y} := \bar{Y}'' \times_{\bar{X}} X \to X$ is the canonical covering of the Enriques surface $X$,
where $X \to \bar{X}$ is the minimal resolution.
\end{itemize}
We do not give $\bar{J}$ explicitly since it will be clear from the equation defining $\bar{Y}'$.
We will describe the type of particular fibers of $J$ according to Lang's classification \cite{Lang:configurations} (for short, we call it the Lang type).

\begin{example}[$12 A_1$, $8 A_1 + D_4^0$, $6 A_1 + D_6^0$, $5 A_1 + E_7^0$] \label{ex:12A1,AD}
The examples with $12 A_1$ ((\ref{case:12A1}) below) 
and $8 A_1 + D_4^0$ ((\ref{case:AD}), $n = 0$) 
are the ones given by Katsura--Kondo \cite{Katsura--Kondo:Enriques finite group}*{Section 3} and Kondo \cite{Kondo:coveredby1}*{Section 3.3} respectively.

	Let $A(t),B(t),C(t) \in k[t]$ be one of the following.
	\begin{enumerate}
		\item \label{case:12A1}  $(A,B,C) = (t^3 (t - 1), t^3 (t-1)^3, 0)$,
		\item \label{case:AD}    $(A,B,C) = (0, t^{3-n} (t-1)^3, n (t-1)^4 )$, $n \in \set{0,2,3}$.
	\end{enumerate}

We have equalities 
$d (A(t)B(t)) / d t = 0$ 
and $C(t) = d (t(t-1)B(t)) / d t$
in each case.

Let $\bar{Y}'$ be the elliptic RDP K3 surface defined by 
\begin{align*} \label{eq:exA:Y'}
y^2  + x y       + t(t-1) A(t) y         + x^3  + t(t-1) B(t) x         &= 0, \\
y'^2 + s^2 x' y' + (1-s) \tilde{A}(s) y' + x'^3 + (1-s) \tilde{B}(s) x' &= 0,
\end{align*}
where $s = t^{-1}$, $x' = t^{-4} x$, $y' = t^{-6} y$, and
\[
\tilde{A}(s) = s^4 B(s^{-1}), \; 
\tilde{B}(s) = s^6 B(s^{-1}), \; 
\tilde{C}(s) = s^4 C(s^{-1}).
\]

	The RDPs of $\bar{Y}'$ and the corresponding singular fibers of the minimal resolution $Y$ are 
\begin{itemize}
	\item[(\ref{case:12A1})] $2 A_9$ ($2 \rI_{10}$) at $t = 0,1$ and $2 A_1$ ($2 \rI_{2}$) at $t = \omega, \omega^2$,
	where $\omega$ and $\omega^2$ are the roots of $t^2 + t + 1 = 0$,
	\item[(\ref{case:AD})] $A_{7-2n}$ ($\rI_{8-2n}$) at $t = 0$, $A_{7}$ ($\rI_{8}$) at $t = 1$, 
	and $D_5^0$ or $D_7^0$ or $E_7^0$ ($\rI_1^*$ or $\rI_3^*$ or $\rIII^*$) at $s = 0$ if $n = 0$ or $n = 2$ or $n = 3$ respectively.
\end{itemize}

Let $\bar{Y}''$ be the elliptic RDP K3 surface which is birational to $\bar{Y}'$ and isomorphic outside the fibers $t = 0,1$,
defined by 
\begin{align*} 
y^2   + x y       + t(t-1) A(t) y         + x^3           + t (t-1) B(t) x        &= 0 \quad (t \neq 0,1), \\
y_1^2 + x_1 y_1   + A(t) y_1              + t (t-1) x_1^3 + B(t) x_1              &= 0, \\
y_2^2 + x_2 y_2   + A(t) x_2^2 y_2        + t (t-1) x_2   + B(t) x_2^3            &= 0, \\
y'^2  + s^2 x' y' + (1-s) \tilde{A}(s) y' + x'^3          + (1-s) \tilde{B}(s) x' &= 0 \quad (s \neq 1),
\end{align*}
where the coordinates are given by
\[
x_1 = \frac{x}{t(t-1)}, \; 
y_1 = \frac{y}{t(t-1)},
\qquad
x = t(t-1) x_1, \;
y = t(t-1) y_1,
\]
\[
x_2 = \frac{t(t-1)}{x}, \; 
y_2 = \frac{t(t-1) y}{x^2}, 
\qquad
x = \frac{t(t-1)}{x_2}, \;
y = \frac{t(t-1) y_2}{x_2^2}.
\]
The RDPs of $\bar{Y}''$ at the fibers $t = 0,1$ are 
\begin{enumerate}
	\item[(\ref{case:12A1})] $A_7 + A_7$ at $(x_1,y_1,t) = (0,0,0),(0,0,1)$ and
$A_1 + A_1$ at $(x_2,y_2,t) = (0,0,0),(0,0,1)$,
	\item[(\ref{case:AD})] 
$A_{5-2n}$ at $(x_1,y_1,t) = (0,0,0)$ (if $n = 0,2$),
$A_5$ at $(x_1,y_1,t) = (0,0,1)$, and
$A_1 + A_1$ at $(x_2,y_2,t) = (0,0,0),(0,0,1)$.
\end{enumerate}
The other fibers remain unchanged.

Let $D_1$ and $D_2$ be the derivations on $\bar{Y}''$ defined as follows,
where $A_t$, $B_t$, and $\tilde{B}_s$ are the derivatives.

\begin{tabular}{lll}
	\toprule
	$-$ & $D_1(-)$ & $D_2(-)$ \\
	\midrule
	$x$   & $0$           & $(t(t-1))^{-1} (x + t(t-1) A(t))$ \\
	$y$   & $t(t-1) C(t)$ & $(t(t-1))^{-1} (y + x^2 + t^2 (t-1)^2 B_t(t))$ \\
	$t$   & $t(t-1)$      & $1$ \\
	\midrule  
	$x_1$ & $x_1$        & $A_t(t)$ \\
	$y_1$ & $y_1 + C(t)$ & $x_1^2 + B_t(t)$ \\
	$t$   & $t(t-1)$     & $1$ \\
	\midrule  
	$x_2$ & $x_2$              & $A_t(t) x_2^2$ \\
	$y_2$ & $y_2 + C(t) x_2^2$ & $1 + B_t(t) x_2^2$ \\
	$t$   & $t(t-1)$           & $1$ \\
	\midrule  
	$x'$  & $0$                   & $(1-s)^{-1} s^2 x' + \tilde{A}(s)$ \\
	$y'$  & $(1-s) \tilde{C}(s) $ & $(1-s)^{-1} (s^2 y' + x'^2 + (1-s)^2 \tilde{B}_s(s))$ \\
	$s$   & $1-s$                 & $s^2$ \\
	\bottomrule
\end{tabular}

In case (\ref{case:12A1}) (resp.\ case (\ref{case:AD}) with $n = 0$), 
the derivations $D_{a,b}$ given by Katsura--Kondo \cite{Katsura--Kondo:Enriques finite group}*{Section 3} (resp.\ Kondo \cite{Kondo:coveredby1}*{Section 3.3}) 
are equal to $ab D_1 + D_2$ (resp.\ $D_1 + (ab)^{-1} D_2$).

Consider the derivation $D = e_1 D_1 + e_2 D_2$ ($e_1, e_2 \in k$).
We observe that $D^2 = e_1 D$ and that
if $(e_1, e_2)$ is generic
(that is, (\ref{case:12A1}) $e_1 - e_2 \neq 0$ and $e_2 \neq 0$, 
and (\ref{case:AD}) $e_1 \neq 0$ and $e_2 \neq 0$)
then $\Fix(D) = \emptyset$.
Therefore, for such $D$, $\bar{X} = \bar{Y}''^D$ is an RDP Enriques surface
with $A_3$, $A_2$, $A_1$, $A_1$ at the images of $A_7$, $A_5$, $D_5^0$, $D_7^0$ respectively
and no other RDPs.
It is supersingular if $e_1 = 0$ in case (\ref{case:12A1}), and classical in all other cases.
Let $X \to \bar{X}$ be the minimal resolution and let $\bar{Y} = \bar{Y}'' \times_{\bar{X}} X$.
Then $\bar{Y}$ is the canonical ($\mu_2$- or $\alpha_2$-) covering of the smooth Enriques surface $X$ with 
(\ref{case:12A1}) $\Sing(\bar{Y}) = 12 A_1$
(\ref{case:AD})   $\Sing(\bar{Y}) = 8 A_1 + D_4^0, 6 A_1 + D_6^0, 5 A_1 + E_7^0$ ($n = 0,2,3$)
respectively.

If $e_1 = 0$ in case (\ref{case:12A1}),
the multiple fiber of $X$ corresponds to the fiber $s = 0$ of $Y$,
which is a supersingular elliptic curve.
In all other cases,
the multiple fibers of $X$ correspond to the fibers $t = \beta_i$ of $Y$,
which are ordinary elliptic curves, 
where $\beta_1, \beta_2$ are the two (distinct) roots of $e_1 t(t-1) + e_2 = 0$
(equivalently, $\beta_1 + \beta_2 = 1$ and $\beta_1 \beta_2 = e_2/e_1$).

In case (\ref{case:AD}),
the singular fiber of additive type (at $s = 0$) of $J$ is of type $\rII$ 
and more precisely it is of Lang type 2A, 2B, 1C for $n = 0,2,3$ respectively.
\end{example}

\begin{example}[$3 D_4^0$, $D_4^0 + D_8^0$, $D_4^0 + E_8^0$] \label{ex:DD}
	Let $A(t),B(t),C(t),G(t) \in k[t]$ be one of the following.
	\begin{enumerate}
		\item \label{case:3D4}  $(A,B,C,G) = (t^2 + t + 1, t^2 + t + 1, t^2, 0)$,
		\item \label{case:D4D8} $(A,B,C,G) = (t + 1, (t + 1)^2, t^2, 0)$,
		\item \label{case:D4E8} $(A,B,C,G) = ((t + 1)^2, (t + 1)^2, (t + 1)^2, t + 1)$.
	\end{enumerate}
	Note that $t^2 C(t) = B(t)^2 + \evenpart{A}(t)$,
	where $\evenpart{A}(t)$ consists of the terms of $A(t)$ of even degree.
	Let $\bar{Y}'$ be the elliptic RDP K3 surface defined by 
	\begin{gather*}
	y^2 + t^2 B(t)^2 y + x^3 + t A(t) x^2 + t^{10} G(t)^2 = 0,
	\\
	y'^2 + \tilde{B}(s)^2 y' + x'^3 + s \tilde{A}(s) x'^2 + \tilde{G}(s)^2 = 0,
	\end{gather*}
	where $s = t^{-1}$, $x' = t^{-4} x$, $y' = t^{-6} y$,
	and \[ \tilde{A}(s) = s^2 A(s^{-1}), \; \tilde{B}(s) = s^2 B(s^{-1}), \; \tilde{C}(s) = s^2 C(s^{-1}), \; \tilde{G}(s) = s G(s^{-1}). \]
	The RDPs of $\bar{Y}'$ and the corresponding singular fibers of the minimal resolution $Y$ are 
	\begin{enumerate}
	\item[(\ref{case:3D4})]  $3 D_5^0$ ($3 \rI_1^*$) at $t = 0, \omega, \omega^2$,
	\item[(\ref{case:D4D8})] $D_5^0$ ($\rI_1^*$) at $t = 0$ and $D_9^0$ ($\rI_5^*$) at $t = 1$,
	\item[(\ref{case:D4E8})] $D_5^0$ ($\rI_1^*$) at $t = 0$ and $E_8^0$ ($\rII^*$) at $t = 1$.
	\end{enumerate}
	Here $\omega$ and $\omega^2$ are the roots of $t^2 + t + 1 = 0$.
	
	Let $\bar{Y}''$ be the elliptic RDP K3 surface which is birational to $\bar{Y}'$ and isomorphic outside the fiber $t = 0$,
	defined by 
\begin{align*} 
	y^2   + t^2 B(t)^2 y      + x^3     + t A(t) x^2          + t^{10} G(t)^2      &= 0 \quad (t \neq 0), \\
	y_0^2 + B(t)^2 x_0^2 y_0  + t^2 x_0 + t A(t) x_0^2        + t^{6} G(t)^2 x_0^4 &= 0, \\
	y'^2  + \tilde{B}(s)^2 y' + x'^3    + s \tilde{A}(s) x'^2 + \tilde{G}(s)^2     &= 0, 
\end{align*}
where the coordinates are given by
	\[
	x_0 = \frac{t^2}{x}, \; 
	y_0 = \frac{t^2 y}{x^2}, 
	\qquad
	x = \frac{t^2}{x_0}, \;
	y = \frac{t^2 y_0}{x_0^2}.
	\]
	Then $\bar{Y}''$ has $D_4^0$ 
	on the fiber $t = 0$, and the RDPs of $\bar{Y}'$ on the other fibers remain unchanged.
	
	Let $D_1$ and $D_2$ be the derivations on $\bar{Y}''$ defined as follows.
	
	\begin{tabular}{lll}
		\toprule
		$-$   & $D_1(-)$ & $D_2(-)$ \\
		\midrule
		$x$   & $t^2 \evenpart{A}(t)$   & $t^2 C(t)$ \\
		$y$   & $0$                     & $t^{-2} x^2$ \\
		$t$   & $t^2$                   & $1$ \\
		\midrule  
		$x_0$ & $\evenpart{A}(t) x_0^2$ & $x_0^2 C(t)$ \\
		$y_0$ & $0$                     & $1$ \\
		$t$   & $t^2$                   & $1$ \\
		\midrule  
		$x'$  & $\evenpart{\tilde{A}}(s)$ & $\tilde{C}(s)$ \\
		$y'$  & $0$                       & $x'^2$ \\
		$s$   & $1$                       & $s^2$ \\
		\bottomrule
	\end{tabular}
	
	Consider the derivation $D = e_1 D_1 + e_2 D_2$ ($e_1, e_2 \in k$).
	We observe that $D^2 = 0$ and that
	if $(e_1, e_2)$ is generic
	(that is, if $e_2 \neq 0$ and $B(\sqrt{e_2/e_1}) \neq 0$) 
	then $\Fix(D) = \emptyset$.
	Therefore, for such $D$, $\bar{X} = \bar{Y}''^D$ is a supersingular RDP Enriques surface
	with $A_1$ at the images of $D_5^0$ and $D_9^0$.
	Let $X \to \bar{X}$ be the minimal resolution and let $\bar{Y} = \bar{Y}'' \times_{\bar{X}} X$.
	Then $\bar{Y}$ is the canonical $\alpha_2$-covering of the smooth supersingular Enriques surface $X$ with
	(\ref{case:3D4})  $\Sing(\bar{Y}) = 3 D_4^0$,
	(\ref{case:D4D8}) $\Sing(\bar{Y}) = D_4^0 + D_8^0$,
	(\ref{case:D4E8}) $\Sing(\bar{Y}) = D_4^0 + E_8^0$.
	
	The multiple fiber of $X$ corresponds to the fiber $t = \sqrt{e_2/e_1}$ of $Y$,
	which is a supersingular elliptic curve.

	The singular fiber at $t = 0$ of $J$ is of type $\rIII$ and of Lang type 10A,
	and the remaining singular fibers are 
	(\ref{case:3D4})  both of type $\rIII$ and of Lang type 10A,
	(\ref{case:D4D8})      of type $\rIII$ and of Lang type 10B,
	(\ref{case:D4E8})      of type $\rII$  and of Lang type 9B.
\end{example}

\begin{example}[$D_{12}^0$] \label{ex:D12}
	Let $\bar{Y}'$ be the elliptic RDP K3 surface defined by 
	\begin{gather*} \label{eq:exD12:Y'}
	y^2 + t^6 y + x^3 + (t^2 + t^6) x + t^7 = 0, \\
	y'^2 + y' + x'^3 + (s^6  + s^2) x' + s^5 = 0,
	\end{gather*}
	where $s = t^{-1}$, $x' = t^{-4} x$, $y' = t^{-6} y$.
	The RDP of $\bar{Y}'$ and the corresponding singular fiber of the minimal resolution $Y$ are $D_{12}^0$ ($\rI_8^*$) at $t = 0$.
	
	Let $\bar{Y} = \bar{Y}''$ be the elliptic RDP K3 surface which is birational to $\bar{Y}'$ and isomorphic outside the fiber $t = 0$,
	defined by 
	\begin{align*} 
	y^2   + t^6 y         + x^3   + (t^2     + t^6) x    + t^7       &= 0 \quad (t \neq 0), \\
	y_0^2 + t^4 x_0^2 y_0 + x_0^3 + t^2 x_0  + t^4 x_0^3 + t^3 x_0^4 &= 0, \\
	y'^2  + y'            + x'^3  + (s^6     + s^2) x'   + s^5       &= 0, 
	\end{align*}
	where the coordinates are given by
	\[
	x_0 = \frac{t^2}{x}, \; 
	y_0 = \frac{t^2 y}{x^2}, 
	\qquad
	x = \frac{t^2}{x_0}, \;
	y = \frac{t^2 y_0}{x_0^2}.
	\]
	The RDP of $\bar{Y}$ is $D_{12}^0$ at $t = x_0 = y_0 = 0$.
	
	Let $D_1$ and $D_2$ be the derivations on $\bar{Y}$ defined as follows.

	\begin{tabular}{lll}
		\toprule
		$-$   & $D_1(-)$ & $D_2(-)$ \\
		\midrule
		$x$   & $0$      & $t^4$                \\
		$y$   & $t^2$    & $t^{-2} (x^2 + t^6)$ \\
		$t$   & $t^2$    & $1$                  \\
		\midrule  
		$x_0$ & $0$      & $t^2 x_0^2 $    \\
		$y_0$ & $x_0^2$  & $t^2 x_0^2 + 1$ \\
		$t$   & $t^2$    & $1$             \\
		\midrule  
		$x'$  & $0$      & $1$          \\
		$y'$  & $s^4$    & $s^2 + x'^2$ \\
		$s$   & $1$      & $s^2$        \\
		\bottomrule
	\end{tabular}
	
	Consider the derivation $D = e_1 D_1 + e_2 D_2$ ($e_1, e_2 \in k$).
	We observe that $D^2 = 0$ and that
	if $(e_1, e_2)$ is generic
	(that is, if $e_2 \neq 0$)
	then $\Fix(D) = \emptyset$.
	Therefore, for such $D$, $X = \bar{Y}^D$ is a supersingular smooth Enriques surface
	and $\bar{Y}$ is its canonical $\alpha_2$-covering with $\Sing(\bar{Y}) = D_{12}^0$.

	The multiple fiber of $X$ corresponds to the fiber $t = \sqrt{e_2/e_1}$ of $Y$,
	which is a supersingular elliptic curve.

	The singular fiber at $t = 0$ of $J$ is of type $\rII$ and of Lang type 9C.
\end{example}

\begin{example}[$D_4^0 + D_8^0$ on the same fiber] \label{ex:D4D8}
	Let $\bar{Y}'$ be the elliptic RDP K3 surface defined by 
	\begin{gather*} 
	y^2 + t^6 y + x^3 + t x^2 = 0,
	\\
	y'^2 + y' + x'^3 + s^3 x'^2 = 0,
	\end{gather*}
	where $s = t^{-1}$, $x' = t^{-4} x$, $y' = t^{-6} y$.
	The RDP of $\bar{Y}'$ and the corresponding singular fiber of the minimal resolution $Y$ are $D_{13}^0$ ($\rI_9^*$) at $t = 0$.

	Let $\bar{Y}''$ be the elliptic RDP K3 surface which is birational to $\bar{Y}'$ and isomorphic outside the fiber $t = 0$,
	defined by 
\begin{align*}
	y^2   + t^6 y         + x^3       + t x^2    &= 0 \quad (t \neq 0), \\
	y_0^2 + t^4 x_0^2 y_0 + t^2 x_0   + t x_0^2  &= 0, \\
	y_3^2 + t^4 y_3       + t^2 x_3^3 + t x_3^2  &= 0, \\
	y'^2  + y           ' + x'^3      + s^3 x'^2 &= 0,
\end{align*}
where the coordinates are given by
	\[
	x_0 = \frac{t^2}{x}, \; 
	y_0 = \frac{t^2 y}{x^2}, 
	\qquad
	x = \frac{t^2}{x_0}, \;
	y = \frac{t^2 y_0}{x_0^2},
	\]
	\[
	x_3 = \frac{x}{t^2}, \;
	y_3 = \frac{y}{t^2}.
	\]
	Then $\bar{Y}''$ has $D_4^0$ at $t = x_0 = y_0 = 0$ and $D_9^0$ at $t = x_3 = y_3 = 0$.
	
	Let $D_1$ and $D_2$ be the derivations on $\bar{Y}''$ defined as follows.

	\begin{tabular}{lll}
		\toprule
		$-$   & $D_1(-)$ & $D_2(-)$ \\
		\midrule
		$x$   & $t^2$   & $t^4$ \\
		$y$   & $0$                     & $t^{-2} x^2$ \\
		$t$   & $t^2$                   & $0$ \\
		\midrule  
		$x_0$ & $x_0^2$ & $t^2 x_0^2$ \\
		$y_0$ & $0$                     & $1$ \\
		$t$   & $t^2$                   & $0$ \\
		\midrule  
		$x_3$ & $1$ & $t^2$ \\
		$y_3$ & $0$ & $x_3^2$ \\
		$t$   & $t^2$ & $0$ \\
		\midrule
		$x'$  & $s^2$ & $1$ \\
		$y'$  & $0$                       & $x'^2$ \\
		$s$   & $1$                       & $0$ \\
		\bottomrule
	\end{tabular}
	
	Consider the derivation $D = e_1 D_1 + e_2 D_2$ ($e_1, e_2 \in k$).
	We observe that $D^2 = 0$ and that
	if $(e_1, e_2)$ is generic
	(that is, if $e_1 \neq 0$ and $e_2 \neq 0$) 
	then $\Fix(D) = \emptyset$.
	Therefore, for such $D$, $\bar{X} = \bar{Y}''^D$ is a supersingular RDP Enriques surface
	with $A_1$ at the image of $D_9^0$.
	Let $X \to \bar{X}$ be the minimal resolution and let $\bar{Y} = \bar{Y}'' \times_{\bar{X}} X$.
	Then $\bar{Y}$ is the canonical $\alpha_2$-covering of the smooth supersingular Enriques surface $X$ with
	$\Sing(\bar{Y}) = D_4^0 + D_8^0$.
	
	The multiple fiber of $X$ corresponds to the fiber $t = 0$ of $Y$.
	In this case this fiber does not move when $D$ vary.
	
	The singular fiber at $t = 0$ of $J$ is of type $\rIII$ and of Lang type 10C.

	We also note that in this example the natural morphism
	$H^0(\bar{Y}, T_{\bar{Y}}) \to H^0(\bP^1, T_{\bP^1})$ is not injective.
\end{example}

\subsection{An example of a canonical covering with an elliptic singularity}

\begin{example}[$\Ella$] \label{ex:Ella}
	This is the example the author gave in \cite{Matsumoto:k3alphap}*{Example 9.4}.

	Let $\bar{Y} \subset \bP^5$ be the intersection of three quadrics
	\begin{align*}
	x_1^2 + x_3^2 + y_1^2         + x_2 y_3 + x_3 y_2 &= 0, \\
	x_2^2         + y_1^2 + y_3^2 + x_1 y_3 + x_3 y_1 &= 0, \\
	y_2^2                         + x_1 y_2 + x_2 y_1 &= 0.
	\end{align*}
	Then it has single singularity at $(x_1, x_2, x_3, y_1, y_2, y_3) = (1, 0, 1, 0, 0, 0)$,
	which is an EDP singularity of type $\Ella$.
	Letting $s^{-1} = t := \frac{x_2}{y_2} = \frac{x_1 + y_2}{y_1}$,
	$\bar{Y}$ admits a structure of an elliptic surface (without assuming the existence of a section) over $\bP^1 = \Spec k[s] \cup \Spec k[t]$.
	It can be written as the intersection of two quadrics in a $\bP^3$-bundle over $\bP^1$ as follows:
	\begin{align*}
	(1 + s^2) x_1^2 + s^4 x_2^2 + x_3^2 + x_2 (s x_3 + y_3) &= 0, \\
	s^2 x_1^2 + (1 + s^4) x_2^2 + s^2 x_2 x_3 + x_1 (s x_3 + y_3) + y_3^2 &= 0 
	\end{align*}
	over $\Spec k[s]$, and 
	\begin{align*}
	(t^2 + 1) y_1^2 + y_2^2 + x_3^2 + x_3 y_2 + t y_2 y_3 &= 0, \\
	y_1^2 + t^2 y_2^2 + y_3^2 + x_3 y_1 + t y_1 y_3 + y_2 y_3 &= 0
	\end{align*}
	over $\Spec k[t]$, 
	glued by  
	\[
	y_1 = s (x_1 + y_2), \;
	y_2 = s x_2, 
	\qquad
	x_1 = t y_1 + y_2, \;
	x_2 = t y_2.
	\]
	The (EDP) singularity is at $s = 0$, $(x_1 : x_2 : x_3 : y_3) = (1 : 0 : 1 : 0)$.	
	
	Let $D_1$ and $D_2$ be the derivations on $\bar{Y}$ defined by 
	\[
	D_1(x_i) = 0, \;
	D_1(y_i) = x_i, 
	\qquad
	D_2(x_i) = y_i, \;
	D_2(y_i) = 0.
	\]
	(To be precise, we consider the derivations taking $\frac{y_j}{x_i}$ to $\frac{D_h(y_j)}{x_i} - y_j \frac{D_h(x_i)}{x_i^2}$, etc.)
	Under the elliptic surface coordinate these derivations are expressed as follows.
	
	\begin{tabular}{lll}
		\toprule
		$-$     & $D_1(-)$ & $D_2(-)$ \\
		\midrule
		$x_1$   & $0$      & $s x_1 + s^2 x_2$                \\
		$x_2$   & $0$      & $s x_2$                \\
		$x_3$   & $0$      & $y_3$                \\
		$y_3$   & $x_3$    & $0$                \\
		$s$     & $1$      & $s^2$                  \\
		\midrule  
		$y_1$   & $t y_1 + y_2$ & $0$ \\
		$y_2$   & $t y_2$       & $0$ \\
		$y_3$   & $x_3$         & $0$ \\
		$x_3$   & $0$           & $y_3$ \\
		$t$     & $t^2$         & $1$ \\
		\bottomrule
	\end{tabular}
	
	Consider the derivation $D = e_1 D_1 + e_2 D_2$ ($e_1, e_2 \in k$).
	We observe that $D^2 = 0$ and that
	if $(e_1, e_2)$ is generic
	(that is, if $e_1 \neq 0$)
	then $\Fix(D) = \emptyset$.
	For such $D$, $X = \bar{Y}^D$ is a supersingular smooth Enriques surface
	and $\bar{Y}$ is its canonical $\alpha_2$-covering with $\Sing(\bar{Y}) = \Ella$.
	
	The multiple fiber of $X$ corresponds to the fiber $t = \sqrt{e_2/e_1}$ of $Y$,
	which is a supersingular elliptic curve.
\end{example}

\subsection*{Acknowledgments}
I thank Hiroyuki Ito, Shigeyuki Kondo, and Stefan Schr\"oer
for helpful comments and discussions.

\end{document}